\documentclass[secthm,seceqn,amsthm,ussrhead,10pt]{amsart}
\usepackage[utf8]{inputenc}
\usepackage[english]{babel}
\usepackage{amssymb,amsmath,amsthm,amsfonts,xcolor,enumerate,hyperref,comment,longtable,cleveref}

\usepackage{times}
\usepackage{cite}
\usepackage{pdflscape}
\usepackage{ulem}
\usepackage[mathcal]{euscript}
\usepackage{tikz}
\usepackage{hyperref}
\usepackage{cancel}
\usepackage{stmaryrd}

\usetikzlibrary{arrows}

\newtheorem{theorem}{Theorem}
\newtheorem{lemma}[theorem]{Lemma}

\newtheorem{remark}[theorem]{Remark}

\newtheorem{definition}[theorem]{Definition}

\usepackage{stmaryrd}
\usepackage{xcolor}

\begin{document}

\noindent{\Large
One-generated nilpotent Novikov algebras}
\footnote{
This work was supported by  
RFBR 18-31-20004; FAPESP  18/15712-0, 	19/00192-3;
Ministerio de Econom\'{i}a y Competitividad (Spain), grant MTM2016-79661-P (European FEDER support included, UE).
The authors thank  Farukh Mashurov   and the referee of the paper for   constructive comments.
}

   \

\

   {\bf  
   Luisa M. Camacho$^{a}$,
   Iqboljon Karimjanov$^{b},$ 
   Ivan   Kaygorodov$^{c}$ \&
   Abror Khudoyberdiyev$^{d}$}

 \

{\tiny

$^{a}$ University of Sevilla, Sevilla, Spain.

$^{b}$ Andijan State University, Andijan, Uzbekistan.

$^{c}$ CMCC, Universidade Federal do ABC, Santo Andr\'e, Brazil.

$^{d}$ National University of Uzbekistan, Institute of Mathematics Academy of
Sciences of Uzbekistan, Tashkent, Uzbekistan.

\smallskip

   E-mail addresses:

\smallskip

Luisa M. Camacho (lcamacho@us.es)

Iqboljon Karimjanov (iqboli@gmail.com)

Ivan   Kaygorodov (kaygorodov.ivan@gmail.com)

Abror Khudoyberdiyev (khabror@mail.ru)

}

\

\noindent{\bf Abstract}: {\it We give a classification of $5$- and  $6$-dimensional complex one-generated nilpotent Novikov algebras.}

\

\

\noindent {\bf Keywords}:
{\it Novikov algebra, nilpotent algebra, algebraic classification, central extension.}

\

\

\noindent {\bf MSC2010}: 17D25, 17A30.

\

\

\section*{Introduction}

The algebraic classification (up to isomorphism) of algebras of dimension $n$ from a certain variety
defined by some family of polynomial identities is a classic problem in the theory of non-associative algebras.
There are many results related to the algebraic classification of low-dimensional algebras in the varieties of
Jordan, Lie, Leibniz, Zinbiel and many other algebras \cite{ack, cfk19, ck13,gkk, gkks, degr3, usefi1, degr1, degr2, ha16, hac18, kv16, krh14}.
Another interesting direction is a study of one-generated objects.
The description of one-generated groups is well-known: there is only one one-generated group of order $n$.
    In the case of algebras, there are some similar results,
    such as the description of $n$-dimensional one-generated nilpotent associative \cite{karel}, noncommutative Jordan \cite{jkk19},  Leibniz and Zinbiel algebras\cite{bakhrom}. 
    It was proven that there is only one $n$-dimensional one-generated nilpotent algebra in these varieties.
    But on the other hand, as we can see in varieties of Novikov \cite{kkl18}, assosymmetric \cite{ikm19},  bicommutative \cite{kpv19}, commutative \cite{fkkv}, and terminal \cite{kkp19geo} algebras, there are more than one $4$-dimensional one-generated nilpotent algebra from these varieties.
    It is easy to see that a nilpotent one-generated algebra
is a central extension of a nilpotent one-generated algebra of a smaller dimension.
In the present paper, we give an algebraic classification of
$5$- and $6$-dimensional one-generated nilpotent Novikov  algebras
introduced by Novikov and Balinsky in 1985 \cite{bn85}.

The variety of Novikov algebras is defined by the following identities:
\[
\begin{array}{rcl}
(xy)z &=& (xz)y,\\
(xy)z-x(yz) &=& (yx)z-y(xz).
\end{array} \]
It contains the commutative associative algebras as a subvariety.
On the other hand, the variety of Novikov algebras is the intersection of
the variety of right commutative algebras (defined by the first Novikov identity)
and
the variety of left symmetric (Pre-Lie) algebras
(defined by the second Novikov identity).
Also, any Novikov algebra under the commutator multiplication is a Lie algebra \cite{bd06, bd08},
and Novikov algebras are related with
Tortken and
Novikov-Poisson algebras  \cite{dz02, xu97, bokut}.
The systematic study of Novikov algebras started after the paper by Zelmanov where
all complex finite-dimensional simple Novikov algebras  were found \cite{ze87}.
The first nontrivial examples of  infinite-dimensional simple  Novikov algebras were constructed in \cite{fi89}.
Also, the simple Novikov algebras were described in infinite-dimensional case and over fields of positive characteristic \cite{os94, xu96, xu01}.
The algebraic classification of $3$-dimensional Novikov algebras was given in \cite{bc01},
and for some classes of $4$-dimensional algebras, it was given in \cite{bg13};
the geometric classification of $3$-dimensional Novikov algebras was given in \cite{bb14}
and of $4$-dimensional nilpotent Novikov algebras in \cite{kkk18}.
Many other pure algebraic properties were studied in a series of papers by Dzhumadildaev  \cite{di14, dt05, dz11, dl02}.

Our method to classify nilpotent Novikov algebras is based on the calculation of central extensions of smaller nilpotent algebras from the same variety.
The algebraic study of central extensions of Lie and non-Lie algebras has been an important topic for years \cite{omirov,ha17,hac16,kkl18,ss78,krs19,zusmanovich}.
First, Skjelbred and Sund used central extensions of Lie algebras for a classification of nilpotent Lie algebras  \cite{ss78}.
After that, using the method described by Skjelbred and Sund,  all  non-Lie central extensions of  all $4$-dimensional Malcev algebras were described \cite{hac16}, and also
all  non-associative central extensions of $3$-dimensional Jordan algebras \cite{ha17},
all  anticommutative central extensions of the $3$-dimensional anticommutative algebras \cite{cfk182},
and all  central extensions of the $2$-dimensional algebras \cite{cfk18}.
Note that the method of central extensions is an important tool in the classification of nilpotent algebras,
which was used to describe
all  $4$-dimensional nilpotent associative algebras \cite{degr1},
all  $5$-dimensional nilpotent Jordan algebras \cite{ha16},
all  $5$-dimensional nilpotent restricted Lie algebras \cite{usefi1},
all  $5$-dimensional nilpotent associative commutative algebras \cite{kpv19},
all  $6$-dimensional nilpotent Lie algebras \cite{degr3,degr2},
all  $6$-dimensional nilpotent Malcev algebras \cite{hac18},
all  $6$-dimensional nilpotent anticommutative algebras \cite{kkl19},
all  $8$-dimensional dual Mock Lie algebras \cite{lisa},
and some others.

\section{The algebraic classification of nilpotent Novikov algebras}
\subsection{Method of classification of nilpotent algebras}
Throughout this paper, we use the notations and methods well written in \cite{ha17,hac16,cfk18},
which we have adapted for the Novikov case with some modifications.
Further in this section we give some important definitions.

Let $({\bf A}, \cdot)$ be a Novikov  algebra over  $\mathbb C$
and $\mathbb V$ a vector space over ${\mathbb C}$. The $\mathbb C$-linear space ${\rm Z^{2}}\left(
\bf A,\mathbb V \right) $ is defined as the set of all  bilinear maps $\theta  \colon {\bf A} \times {\bf A} \longrightarrow {\mathbb V}$ such that
\[ \theta(xy,z)=\theta(xz,y), \]
\[ \theta(xy,z)-\theta(x,yz)= \theta(yx,z)-\theta(y,xz). \]
These elements will be called {\it cocycles}. For a
linear map $f$ from $\bf A$ to  $\mathbb V$, if we define $\delta f\colon {\bf A} \times
{\bf A} \longrightarrow {\mathbb V}$ by $\delta f  (x,y ) =f(xy )$, then $\delta f\in {\rm Z^{2}}\left( {\bf A},{\mathbb V} \right) $. We define ${\rm B^{2}}\left({\bf A},{\mathbb V}\right) =\left\{ \theta =\delta f\ : f\in {\rm Hom}\left( {\bf A},{\mathbb V}\right) \right\} $.
We define the {\it second cohomology space} ${\rm H^{2}}\left( {\bf A},{\mathbb V}\right) $ as the quotient space ${\rm Z^{2}}
\left( {\bf A},{\mathbb V}\right) \big/{\rm B^{2}}\left( {\bf A},{\mathbb V}\right) $.

\

Let $\operatorname{Aut}({\bf A}) $ be the automorphism group of  ${\bf A} $ and let $\phi \in \operatorname{Aut}({\bf A})$. For $\theta \in
{\rm Z^{2}}\left( {\bf A},{\mathbb V}\right) $ define  the action of the group $\operatorname{Aut}({\bf A}) $ on ${\rm H^{2}}\left( {\bf A},{\mathbb V}\right) $ by $\phi \theta (x,y)
=\theta \left( \phi \left( x\right) ,\phi \left( y\right) \right) $.  It is easy to verify that
 ${\rm B^{2}}\left( {\bf A},{\mathbb V}\right) $ is invariant under the action of $\operatorname{Aut}({\bf A}).$  
 So, we have an induced action of  $\operatorname{Aut}({\bf A})$  on ${\rm H^{2}}\left( {\bf A},{\mathbb V}\right)$.

\

Let $\bf A$ be a Novikov  algebra of dimension $m$ over  $\mathbb C$ and ${\mathbb V}$ be a $\mathbb C$-vector
space of dimension $k$. For $\theta \in {\rm Z^{2}}\left(
{\bf A},{\mathbb V}\right) $, define on the linear space ${\bf A}_{\theta } = {\bf A}\oplus {\mathbb V}$ the
bilinear product `` $\left[ -,-\right] _{{\bf A}_{\theta }}$'' by $\left[ x+x^{\prime },y+y^{\prime }\right] _{{\bf A}_{\theta }}=
 xy +\theta(x,y) $ for all $x,y\in {\bf A},x^{\prime },y^{\prime }\in {\mathbb V}$.
The algebra ${\bf A}_{\theta }$ is called an $k$-{\it dimensional central extension} of ${\bf A}$ by ${\mathbb V}$. One can easily check that ${\bf A_{\theta}}$ is a Novikov 
algebra if and only if $\theta \in {\rm Z^2}({\bf A}, {\mathbb V})$.

Call the
set $\operatorname{Ann}(\theta)=\left\{ x\in {\bf A}:\theta \left( x, {\bf A} \right)+ \theta \left({\bf A} ,x\right) =0\right\} $
the {\it annihilator} of $\theta $. We recall that the {\it annihilator} of an  algebra ${\bf A}$ is defined as
the ideal $\operatorname{Ann}(  {\bf A} ) =\left\{ x\in {\bf A}:  x{\bf A}+ {\bf A}x =0\right\}$. Observe
 that
$\operatorname{Ann}\left( {\bf A}_{\theta }\right) =(\operatorname{Ann}(\theta) \cap\operatorname{Ann}({\bf A}))
 \oplus {\mathbb V}$.

\

The following result shows that every algebra with a non-zero annihilator is a central extension of a smaller-dimensional algebra.
As a consequence, it is easy to see that
a one-generated nilpotent algebra is a central extension of a one-generated
nilpotent algebra of a smaller dimension.

\begin{lemma}
Let ${\bf A}$ be an $n$-dimensional Novikov algebra such that $\dim (\operatorname{Ann}({\bf A}))=m\neq0$. Then there exists, up to isomorphism, a unique $(n-m)$-dimensional Novikov  algebra ${\bf A}'$ and a bilinear map $\theta \in {\rm Z^2}({\bf A}, {\mathbb V})$ with $\operatorname{Ann}({\bf A})\cap\operatorname{Ann}(\theta)=0$, where $\mathbb V$ is a vector space of dimension m, such that ${\bf A} \cong {{\bf A}'}_{\theta}$ and
 ${\bf A}/\operatorname{Ann}({\bf A})\cong {\bf A}'$.
\end{lemma}

\begin{proof}
Let ${\bf A}'$ be a linear complement of $\operatorname{Ann}({\bf A})$ in ${\bf A}$. Define a linear map $P \colon {\bf A} \longrightarrow {\bf A}'$ by $P(x+v)=x$ for $x\in {\bf A}'$ and $v\in\operatorname{Ann}({\bf A})$, and define a multiplication on ${\bf A}'$ by $[x, y]_{{\bf A}'}=P(x y)$ for $x, y \in {\bf A}'$.
For $x, y \in {\bf A}$, we have
\[P(xy)=P((x-P(x)+P(x))(y- P(y)+P(y)))=P(P(x) P(y))=[P(x), P(y)]_{{\bf A}'}. \]

Since $P$ is a homomorphism $P({\bf A})={\bf A}'$ is a Novikov algebra and
 ${\bf A}/\operatorname{Ann}({\bf A})\cong {\bf A}'$, which gives us the uniqueness. Now, define the map $\theta \colon {\bf A}' \times {\bf A}' \longrightarrow\operatorname{Ann}({\bf A})$ by $\theta(x,y)=xy- [x,y]_{{\bf A}'}$.
  Thus, ${\bf A}'_{\theta}$ is ${\bf A}$ and therefore $\theta \in {\rm Z^2}({\bf A}, {\mathbb V})$ and $\operatorname{Ann}({\bf A})\cap\operatorname{Ann}(\theta)=0$.
\end{proof}

\

\;
\begin{definition}
Let ${\bf A}$ be an algebra and $I$ be a subspace of $\operatorname{Ann}({\bf A})$. If ${\bf A}={\bf A}_0 \oplus I$
then $I$ is called an {\it annihilator component} of ${\bf A}$.
\end{definition}
\begin{definition}
A central extension of an algebra $\bf A$ without annihilator component is called a {\it non-split central extension}.
\end{definition}

Our task is to find all central extensions of an algebra $\bf A$ by
a space ${\mathbb V}$.  In order to solve the isomorphism problem we need to study the
action of $\operatorname{Aut}({\bf A})$ on ${\rm H^{2}}\left( {\bf A},{\mathbb V}
\right) $. To do that, let us fix a basis $e_{1},\ldots ,e_{s}$ of ${\mathbb V}$, and $
\theta \in {\rm Z^{2}}\left( {\bf A},{\mathbb V}\right) $. Then $\theta $ can be uniquely
written as $\theta \left( x,y\right) =
\displaystyle \sum_{i=1}^{s} \theta _{i}\left( x,y\right) e_{i}$, where $\theta _{i}\in
{\rm Z^{2}}\left( {\bf A},\mathbb C\right) $. Moreover, $\operatorname{Ann}(\theta)=\operatorname{Ann}(\theta _{1})\cap\operatorname{Ann}(\theta _{2})\cap\cdots \cap\operatorname{Ann}(\theta _{s})$. Furthermore, $\theta \in
{\rm B^{2}}\left( {\bf A},{\mathbb V}\right) $\ if and only if all $\theta _{i}\in {\rm B^{2}}\left( {\bf A},
\mathbb C\right) $.
It is not difficult to prove (see \cite[Lemma 13]{hac16}) that given a Novikov algebra ${\bf A}_{\theta}$, if we write as
above $\theta \left( x,y\right) = \displaystyle \sum_{i=1}^{s} \theta_{i}\left( x,y\right) e_{i}\in {\rm Z^{2}}\left( {\bf A},{\mathbb V}\right) $ and 
$\operatorname{Ann}(\theta)\cap \operatorname{Ann}\left( {\bf A}\right) =0$, then ${\bf A}_{\theta }$ has an
annihilator component if and only if $\left[ \theta _{1}\right] ,\left[
\theta _{2}\right] ,\ldots ,\left[ \theta _{s}\right] $ are linearly
dependent in ${\rm H^{2}}\left( {\bf A},\mathbb C\right) $.

\;

Let ${\mathbb V}$ be a finite-dimensional vector space over $\mathbb C$. The {\it Grassmannian} $G_{k}\left( {\mathbb V}\right) $ is the set of all $k$-dimensional
linear subspaces of $ {\mathbb V}$. Let $G_{s}\left( {\rm H^{2}}\left( {\bf A},\mathbb C\right) \right) $ be the Grassmannian of subspaces of dimension $s$ in
${\rm H^{2}}\left( {\bf A},\mathbb C\right) $. There is a natural action of $\operatorname{Aut}({\bf A})$ on $G_{s}\left( {\rm H^{2}}\left( {\bf A},\mathbb C\right) \right) $.
Let $\phi \in \operatorname{Aut}({\bf A})$. For $W=\left\langle
\left[ \theta _{1}\right] ,\left[ \theta _{2}\right] ,\dots,\left[ \theta _{s}
\right] \right\rangle \in G_{s}\left( {\rm H^{2}}\left( {\bf A},\mathbb C
\right) \right) $ define $\phi W=\left\langle \left[ \phi \theta _{1}\right]
,\left[ \phi \theta _{2}\right] ,\dots,\left[ \phi \theta _{s}\right]
\right\rangle $. We denote the orbit of $W\in G_{s}\left(
{\rm H^{2}}\left( {\bf A},\mathbb C\right) \right) $ under the action of $\operatorname{Aut}({\bf A})$ by $\operatorname{Orb}(W)$. Given
\[
W_{1}=\left\langle \left[ \theta _{1}\right] ,\left[ \theta _{2}\right] ,\dots,
\left[ \theta _{s}\right] \right\rangle ,W_{2}=\left\langle \left[ \vartheta
_{1}\right] ,\left[ \vartheta _{2}\right] ,\dots,\left[ \vartheta _{s}\right]
\right\rangle \in G_{s}\left( {\rm H^{2}}\left( {\bf A},\mathbb C\right)
\right),
\]
we easily have that if $W_{1}=W_{2}$, then $ \bigcap\limits_{i=1}^{s}\operatorname{Ann}(\theta _{i})\cap \operatorname{Ann}\left( {\bf A}\right) = \bigcap\limits_{i=1}^{s}
\operatorname{Ann}(\vartheta _{i})\cap\operatorname{Ann}( {\bf A}) $, and therefore we can introduce
the set
\[
{\bf T}_{s}({\bf A}) =\left\{ W=\left\langle \left[ \theta _{1}\right] ,
\left[ \theta _{2}\right] ,\dots,\left[ \theta _{s}\right] \right\rangle \in
G_{s}\left( {\rm H^{2}}\left( {\bf A},\mathbb C\right) \right) : \bigcap\limits_{i=1}^{s}\operatorname{Ann}(\theta _{i})\cap\operatorname{Ann}({\bf A}) =0\right\},
\]
which is stable under the action of $\operatorname{Aut}({\bf A})$.

\

Now, let ${\mathbb V}$ be an $s$-dimensional linear space and let us denote by
${\bf E}\left( {\bf A},{\mathbb V}\right) $ the set of all {\it non-split $s$-dimensional central extensions} of ${\bf A}$ by
${\mathbb V}$. By above, we can write
\[
{\bf E}\left( {\bf A},{\mathbb V}\right) =\left\{ {\bf A}_{\theta }:\theta \left( x,y\right) = \sum_{i=1}^{s}\theta _{i}\left( x,y\right) e_{i} \ \ \text{and} \ \ \left\langle \left[ \theta _{1}\right] ,\left[ \theta _{2}\right] ,\dots,
\left[ \theta _{s}\right] \right\rangle \in {\bf T}_{s}({\bf A}) \right\} .
\]
We also have the following result, which can be proved as in \cite[Lemma 17]{hac16}.

\begin{lemma}
 Let ${\bf A}_{\theta },{\bf A}_{\vartheta }\in {\bf E}\left( {\bf A},{\mathbb V}\right) $. Suppose that $\theta \left( x,y\right) =  \displaystyle \sum_{i=1}^{s}
\theta _{i}\left( x,y\right) e_{i}$ and $\vartheta \left( x,y\right) =
\displaystyle \sum_{i=1}^{s} \vartheta _{i}\left( x,y\right) e_{i}$.
Then the Novikov algebras ${\bf A}_{\theta }$ and ${\bf A}_{\vartheta } $ are isomorphic
if and only if
$$\operatorname{Orb}\left\langle \left[ \theta _{1}\right] ,
\left[ \theta _{2}\right] ,\dots,\left[ \theta _{s}\right] \right\rangle =
\operatorname{Orb}\left\langle \left[ \vartheta _{1}\right] ,\left[ \vartheta
_{2}\right] ,\dots,\left[ \vartheta _{s}\right] \right\rangle .$$
\end{lemma}

This shows that there exists a one-to-one correspondence between the set of $\operatorname{Aut}({\bf A})$-orbits on ${\bf T}_{s}\left( {\bf A}\right) $ and the set of
isomorphism classes of ${\bf E}\left( {\bf A},{\mathbb V}\right) $. Consequently we have a
procedure that allows us, given a Novikov algebra ${\bf A}'$ of
dimension $n-s$, to construct all non-split central extensions of ${\bf A}'$. This procedure is:

\; \;

{\centerline {\textsl{Procedure}}}

\begin{enumerate}
\item For a given Novikov algebra ${\bf A}'$ of dimension $n-s $, determine ${\rm H^{2}}( {\bf A}',\mathbb {C}) $, $\operatorname{Ann}({\bf A}')$ and $\operatorname{Aut}({\bf A}')$.

\item Determine the set of $\operatorname{Aut}({\bf A}')$-orbits on $T_{s}({\bf A}') $.

\item For each orbit, construct the Novikov algebra associated with a
representative of it.
\end{enumerate}

\subsection{Notations}

Let ${\bf A}$ be a Novikov algebra with
a basis $e_{1},e_{2},\dots,e_{n}$. Then by $\Delta _{ij}$\ we will denote the bilinear form
$\Delta _{ij} \colon {\bf A}\times {\bf A}\longrightarrow \mathbb C$
with $\Delta _{ij}\left( e_{l},e_{m}\right) = \delta_{il}\delta_{jm}$.
Then the set $\left\{ \Delta_{ij}:1\leq i, j\leq n\right\} $ is a basis for the linear space of
the bilinear forms on ${\bf A}$. Then every bilinear form $\theta$ can be uniquely written as $
\theta = \displaystyle \sum_{1\leq i,j\leq n} c_{ij}\Delta _{{i}{j}}$, where $
c_{ij}\in \mathbb C$.
Let us fix the following notations:
$$\begin{array}{lll}
{\mathcal N}^i_j& \mbox{---}& j\mbox{th }i\mbox{-dimensional one-generated Novikov algebra.} \\
\end{array}$$

\subsection{Low dimensional one-generated nilpotent Novikov algebras.}
Thanks to \cite{kkk18} we have the algebraic classification of $2$-, $3$- and $4$- dimensional 
and $5$-dimensional with $2$-dimensional annihilator one-generated nilpotent algebras.
Let us give a list if $1$-, $2$-, $3$- and $4$-dimensional one-generated nilpotent Novikov algebras.
\begin{longtable}{llllllllll}
${\mathcal N}^2_{01}$ &$:$ & $e_1 e_1 = e_2$\\
${\mathcal N}^3_{01}$ &$:$ & $e_1 e_1 = e_2$ &    $e_2 e_1=e_3$ \\
${\mathcal N}^3_{02}(\lambda)$  &$:$ & $e_1 e_1 = e_2$ &   $e_1 e_2=e_3$ &   $e_2 e_1=\lambda e_3$\\
${\mathcal N}^4_{01}$& $:$ &$e_1 e_1 = e_2$ &  $ e_1 e_2=e_3$ &  $e_2 e_1=e_4$\\
${\mathcal N}^4_{02}$ &$:$ & $e_1 e_1 = e_2$ &    $e_2 e_1=e_3$ &    $e_1 e_3=e_4$ &   $e_3 e_1=-e_4$\\
${\mathcal N}^4_{03}(\lambda)$ &$:$ & $e_1 e_1 = e_2$ &    $e_1 e_2=e_3$ &   $e_1 e_3=(2-\lambda)e_4$ & $e_2 e_1= \lambda e_3$ &  $ e_2 e_2=\lambda e_4$ &   $e_3 e_1=\lambda e_4$\\
${\mathcal N}^4_{04}$ &$:$&
$e_1 e_1 = e_2$ &    $e_1 e_2=e_3$ &    $e_1 e_3=2e_4$ &  $e_2 e_1= e_4$\\
${\mathcal N}^4_{05}$ &$:$ &
$e_1 e_1 = e_2$ &   $e_1 e_2=e_3$ &  $e_1 e_3=e_4$ & $ e_2 e_1= e_3 + e_4$ &   $e_2 e_2= e_4$ &   $e_3 e_1= e_4$
\end{longtable}

\begin{remark}
Note that a non-split central extension of a split algebra can not be a one-generated algebra.
Hence, we will consider central extensions only of non-split one-generated nilpotent algebras.
\end{remark}

\section{Classification of 5-dimensional one-generated nilpotent Novikov algebras}
All the necessary information about
coboundaries, cocycles and  second cohomology spaces of $5$-dimensional one-generated nilpotent Novikov algebras is given in Table A (see Table A from Appendix).

From the results of \cite{kkk18}, it is not difficult to see that there are two $5$-dimensional one-generated Novikov algebras with $2$-dimensional annihilator:
\begin{longtable}{llllllll}
${\mathcal N}^5_{01}$ &:& $e_1 e_1 = e_2$  & $e_1e_2=e_4$ & $e_1e_3=e_5$ & $e_2 e_1=e_3$ & $e_3e_1=-e_5$&   \\
${\mathcal N}^5_{02}(\lambda)$ &:& $e_1 e_1 = e_2 $ & $e_1 e_2=e_3$ & $e_1e_3=(2-\lambda)e_5$ &  $e_2 e_1=\lambda e_3+e_4$&  $e_2e_2=\lambda e_5$ & $e_3e_1=\lambda e_5$
\end{longtable}

\subsection{Central extensions of ${\mathcal N}^4_{01}$}

Let us use the following notations
\[\nabla_1=[\Delta_{13}], \quad  \nabla_2=[\Delta_{14}] - [\Delta_{41}], \quad \nabla_3 = [\Delta_{22}]+[\Delta_{31}]+2[\Delta_{41}].\]

The automorphism group of ${\mathcal N}^4_{01}$ consists of invertible matrices of the form
\[\phi=\begin{pmatrix}
x & 0 & 0 & 0 \\
               y & x^2 & 0 & 0 \\
               z & xy & x^3 & 0 \\
               v & xy & 0 & x^3 \\
\end{pmatrix}. \]

Since
\[ \phi^T\begin{pmatrix}
0 & 0 & \alpha_1 & \alpha_2 \\
               0 & \alpha_3 & 0 & 0 \\
               \alpha_3 & 0 & 0 & 0 \\
               2\alpha_3-\alpha_2 & 0 & 0 & 0
\end{pmatrix}\phi =
\begin{pmatrix}
 \alpha^* & \alpha^{**} & \alpha_1^* & \alpha_2^* \\
               \alpha^{***} & \alpha_3^* & 0 & 0 \\
               \alpha_3^* & 0 & 0 & 0 \\
               2\alpha_3^*-\alpha_2^* & 0 & 0 & 0
\end{pmatrix},
\]
 we have that the action of ${\rm Aut} ({\mathcal N}^4_{01})$ on the subspace
$\langle \sum\limits_{i=1}^3\alpha_i\nabla_i  \rangle$
is given by
$\langle \sum\limits_{i=1}^3\alpha_i^{*}\nabla_i\rangle,$
where
\[\alpha_1^{*}=x^4\alpha_1, \quad \alpha_2^{*}=x^4\alpha_2, \quad \alpha_3^{*}=x^4 \alpha_3.\]

\subsubsection{1-dimensional central extensions of ${\mathcal N}^4_{01}$}
Since we are interested only in new algebras, we have the following cases:
\begin{enumerate}
\item if $\alpha_3\neq 0,$ then we have the representative $\langle \lambda \nabla_1 +\mu \nabla_2 +  \nabla_3 \rangle;$
\item if $\alpha_3=0,$ then $(\alpha_1,\alpha_2)\neq (0, 0)$ and we have the representative
$\langle \nabla_1+\lambda\nabla_2\rangle$ with $\lambda \neq 0.$
\end{enumerate}

Hence, we have the following new 5-dimensional algebras:
\begin{longtable}{lllllllllll}
${\mathcal N}^5_{03}(\lambda, \mu)$ & :&
$e_1 e_1 = e_2$ & $e_1 e_2=e_3$&  $e_1e_3=\lambda e_5$& $e_1 e_4= \mu e_5$& &&&\\
&& $e_2 e_1=e_4$ & $e_2 e_2= e_5$& $e_3e_1= e_5$ & $e_4 e_1= (2- \mu) e_5$  \\
${\mathcal N}^5_{04}(\lambda)_{\lambda\neq0}$ & :&
$e_1 e_1 = e_2$ & $e_1 e_2=e_3$&  $e_1e_3= e_5$& $e_1 e_4= \lambda e_5$ &   $e_2 e_1=e_4$ & $e_4e_1=-\lambda e_5$ & & 
\end{longtable}

\subsubsection{2-dimensional central extensions of ${\mathcal N}^4_{01}$}

Consider the vector space generated by the following two cocycles
$$\begin{array}{rcl}
\theta_1 &=& \alpha_1 \nabla_1+\alpha_2\nabla_2+\alpha_3\nabla_3,  \\
\theta_2 &=& \beta_1 \nabla_1+\beta_2\nabla_2.
\end{array}$$

If $\alpha_3=0,$ then we have the representative $\langle \nabla_1, \nabla_2 \rangle.$
If $\alpha_3\neq 0,$ then

\begin{enumerate}
    \item $\beta_2=0,$ then we have the family of representatives
    $\langle \nabla_1, \lambda \nabla_2+\nabla_3 \rangle;$

    \item $\beta_2\neq 0,$  then we have the family of   representatives
    $\langle \lambda \nabla_1+\nabla_2, \mu \nabla_1 + \nabla_3 \rangle.$
\end{enumerate}

Hence, we have the following 6-dimensional algebras:
\begin{longtable}{lllllllll}
 ${\mathcal N}^6_{01}$ & : &
$e_1 e_1 = e_2$ &    $e_1e_2=e_3$ & $e_1 e_3=e_5$ & $e_1e_4=e_6$ & $e_2 e_1=e_4$ & $e_4 e_1=-e_6$ \\

 ${\mathcal N}^6_{02}(\lambda)$ & : &
$e_1 e_1 = e_2$&   $e_1 e_2=e_3$ &    $e_1 e_3=e_5$ & $e_1e_4=\lambda e_6$ &  && \\
&& $e_2 e_1= e_4$ & $e_2e_2=  e_6$ & $e_3e_1= e_6$  & $e_4e_1= (2-\lambda)e_6$ &&& \\

 ${\mathcal N}^6_{03}(\lambda, \mu)$ & : &
$e_1 e_1 = e_2$ &  $e_1 e_2=e_3$ &    $e_1 e_3=\lambda e_5 + \mu e_6$ & $e_1e_4= e_5$ &  && \\
&& $e_2 e_1= e_4$ & $e_2e_2=  e_6$ & $e_3e_1= e_6$  & $e_4e_1= -e_5+2e_6$ & &&
\end{longtable}

\subsection{Central extensions of ${\mathcal N}^4_{02}$} Let us use the following notations
\[\nabla_1=[\Delta_{12}], \quad \nabla_2=2[\Delta_{14}]-[\Delta_{23}]-[\Delta_{41}].\]

The automorphism group of ${\mathcal N}^4_{02}$ consists of invertible matrices of the form
\[\phi=\begin{pmatrix}
x & 0 & 0  & 0\\
0 & x^2 & 0 & 0\\
y & 0 & x^3 & 0 \\
z & 0 & 0 &  x^4\\
\end{pmatrix}. \]

Since
\[ \phi^T\begin{pmatrix}
0 & \alpha_1 & 0 &  2\alpha_2\\
0 & 0 &  -\alpha_2& 0  \\
0 & 0 & 0 & 0 \\
 -\alpha_2 & 0 & 0 & 0
\end{pmatrix}\phi =
\begin{pmatrix}
\alpha^* & \alpha^*_1 & 0& 2\alpha_2^* \\
\alpha^{**}&0 &-\alpha_2^*  & 0 \\
0 & 0 & 0 & 0 \\
-\alpha_2^*  & 0 &0&0
\end{pmatrix},
\]
 we have that the action of ${\rm Aut} ({\mathcal N}^4_{02})$ on the subspace
$\langle \sum\limits_{i=1}^2\alpha_i\nabla_i \rangle$
is given by
$\langle \sum\limits_{i=1}^2 \alpha_i^* \nabla_i \rangle,$
where
\[\alpha_1^*=x^3\alpha_1, \quad \alpha_2^* =x^5\alpha_2. \]

\subsubsection{1-dimensional central extensions of ${\mathcal N}^4_{02}$}
Since we are interested only in new algebras, we have assume $\alpha_2 \neq0,$ and we have the following cases:
\begin{enumerate}
\item if $\alpha_1=0,$ then we have the representative $\langle \nabla_2\rangle;$

\item if $\alpha_1 \neq 0,$ then by choosing
$x=\sqrt{ \alpha_1 / \alpha_2},$   we have the representative $\langle\nabla_1+\nabla_2\rangle.$
\end{enumerate}

Thus, we have the following $5$-dimensional algebras:
\begin{longtable}{lllllllllll}
 ${\mathcal N}^5_{05}$ &:&
$e_1 e_1 = e_2$ &    $e_1 e_3=e_4$ & $e_1e_4=2e_5$ & $e_2 e_1=e_3$ &   $e_2e_3=-e_5$ &  $e_3 e_1=-e_4$  & $e_4e_1=-e_5$ & \\
 ${\mathcal N}^5_{06}$ & : &
$e_1 e_1 = e_2$ &    $e_1e_2=e_5$ & $e_1 e_3=e_4$ & $e_1e_4=2e_5$ &  $e_2 e_1=e_3$ & $e_2e_3=-e_5$ &  $e_3 e_1=-e_4$  & $e_4e_1=-e_5$ 
\end{longtable}

\subsubsection{2-dimensional central extensions of ${\mathcal N}^4_{02}$}
Since $\dim\left({\rm H^2}( {\mathcal N}^4_{02})\right)=2,$ there is only one 2-dimensional central extension of  ${\mathcal N}^4_{02}.$
Thus, we have the following $6$-dimensional algebra:
\begin{longtable}{lllllllllll}
 ${\mathcal N}^6_{04}$ & : &
$e_1 e_1 = e_2$ &    $e_1e_2=e_5$ & $e_1 e_3=e_4$ & $e_1e_4=2e_6$& $e_2 e_1=e_3$ & $e_2e_3=-e_6$  & $e_3 e_1=-e_4$  & $e_4e_1=-e_6$ 
\end{longtable}

\subsection{Central extensions of ${\mathcal N}^4_{03}(\lambda)_{\lambda\notin\{0, 1\}}$} Let us use the following notations
\[\nabla_1=[\Delta_{21}], \quad \nabla_2=(3-2\lambda)[\Delta_{14}]+\lambda(2-\lambda)[\Delta_{23}]+\lambda[\Delta_{32}]+\lambda[\Delta_{41}].\]

The automorphism group of ${\mathcal N}^4_{03}(\lambda)$ consists of invertible matrices of the form
\[\phi=\begin{pmatrix}
x & 0 & 0  & 0\\
0 & x^2 & 0 & 0\\
y & 0 & x^3 & 0 \\
z & 2xy & 0 &  x^4\\
\end{pmatrix}. \]
Since

\[ \phi^T\begin{pmatrix}
0 & 0 & 0 &  (3-2\lambda)\alpha_2\\
\alpha_1 & 0 & \lambda( 2-\lambda)\alpha_2& 0  \\
0 & \lambda\alpha_2 & 0 & 0 \\
 \lambda\alpha_2 & 0 & 0 & 0
\end{pmatrix}\phi =
\begin{pmatrix}
\alpha^{**} & \alpha^* & 0& (3-2\lambda)\alpha_2^* \\
\alpha_1^*+\lambda\alpha^* &0 &\lambda(2-\lambda)\alpha_2^*  & 0 \\
0 & \lambda\alpha_2^* & 0 & 0 \\
\lambda\alpha_2^*& 0 &0&0
\end{pmatrix},
\]
 we have that the action of ${\rm Aut} ({\mathcal N}^4_{03}(\lambda))$ on the subspace
$\langle \sum\limits_{i=1}^2 \alpha_i\nabla_i  \rangle$
is given by
$\langle \sum\limits_{i=1}^2 \alpha_i^* \nabla_i  \rangle,$
where
\[\alpha_1^* =x^2(x\alpha_1+2\lambda(\lambda-1)y\alpha_2), \quad \alpha_2^* =x^5\alpha_2. \]

\subsubsection{1-dimensional central extensions of ${\mathcal N}^4_{03}(\lambda)_{\lambda\notin\{0, 1\}}$}
Since we are interested only in new algebras, we have $\alpha_2 \neq0$.
Then choosing
$y=\frac{x\alpha_1}{2\lambda(1-\lambda)\alpha_2},$  we have the representative $\langle\nabla_2\rangle.$
Hence, we have the following $5$-dimensional algebras:
\begin{longtable}{lllllll}
 ${\mathcal N}^5_{07}(\lambda)_{\lambda\notin\{0, 1\}}$ & : &
$e_1 e_1 = e_2$&   $e_1 e_2=e_3$&    $e_1 e_3=(2-\lambda)e_4$ & $e_1e_4=(3-2\lambda) e_5$ & $e_2 e_1= \lambda e_3$\\
&& $e_2e_2= \lambda e_4$ & $e_2e_3=\lambda (2-\lambda )e_5$ & $e_3e_1=\lambda e_4$ & $e_3e_2=\lambda e_5$ &  $e_4e_1=\lambda e_5$
\end{longtable}

\subsubsection{2-dimensional central extensions of ${\mathcal N}^4_{03}(\lambda)_{\lambda\notin\{0, 1\}}$}
Since $\dim\left({\rm H^2}( {\mathcal N}^4_{03}(\lambda)_{\lambda\neq0; 1})\right)=2,$ there is only one 2-dimensional central extension of  ${\mathcal N}^4_{03}(\lambda)_{\lambda\notin\{0, 1\}}.$
Hence, we have the following $6$-dimensional algebras:
\begin{longtable}{lllllll}
 ${\mathcal N}^6_{05}(\lambda)_{\lambda\notin\{0, 1\}}$ & : & $e_1 e_1 = e_2$ &
$e_1 e_2=e_3$ &    $e_1 e_3=(2-\lambda)e_4$ & $e_1e_4=(3-2\lambda)
e_5$ &
$e_2 e_1= \lambda e_3+e_6$  \\
&&$e_2e_2= \lambda e_4$ & $e_2e_3=\lambda (2-\lambda )e_5$ & $e_3e_1=\lambda e_4$ & $e_3e_2=\lambda e_5$ &  $e_4e_1=\lambda e_5$
\end{longtable}

\subsection{Central extensions of ${\mathcal N}^4_{03}(0)$.} Let us use the following notations
\[\nabla_1=[\Delta_{14}], \quad \nabla_2=[\Delta_{21}], \quad \nabla_3=2[\Delta_{23}]-2[\Delta_{32}]+[\Delta_{41}].\]

The automorphism group of ${\mathcal N}^4_{03}(0)$ consists of invertible matrices of the form
\[\phi=\begin{pmatrix}
x & 0 & 0  & 0\\
y & x^2 & 0 & 0\\
z & xy & x^3 & 0 \\
t & 2xz & 2x^2y &  x^4\\
\end{pmatrix}. \]

Since
\[ \phi^T\begin{pmatrix}
0 & 0 & 0 &  \alpha_1\\
\alpha_2 & 0 &  2\alpha_3& 0  \\
0 & -2\alpha_3 & 0 & 0 \\
 \alpha_3 & 0 & 0 & 0
\end{pmatrix}\phi =
\begin{pmatrix}
\alpha^* & \alpha^{**} & \alpha^{***}& \alpha_1^* \\
\alpha_2^* &0 &2\alpha_3^*  & 0 \\
0 & -2 \alpha_3^* & 0 & 0 \\
\alpha_3^*& 0 &0&0
\end{pmatrix},
\]
 we have that the action of ${\rm Aut} ({\mathcal N}^4_{03}(0))$ on the subspace
$\langle \sum\limits_{i=1}^3 \alpha_i\nabla_i  \rangle$
is given by
$\langle \sum\limits_{i=1}^3 \alpha_i^* \nabla_i  \rangle,$
where
\[\alpha_1^* =x^5\alpha_1, \quad \alpha_2^* =x^3\alpha_2+(4x^2z-2xy^2)\alpha_3, \quad \alpha_3^* =x^5\alpha_3. \]

\subsubsection{1-dimensional central extensions of ${\mathcal N}^4_{03}(0)$} 
Since we are interested only in new algebras, we have $(\alpha_1, \alpha_3) \neq (0,0)$ and:
\begin{enumerate}
\item if $\alpha_3 = 0, \alpha_2 =0,$ then we have the representative $\langle\nabla_1\rangle;$

\item if $ \alpha_3=0, \alpha_2 \neq0,$ then choosing
$x=\sqrt{ {\alpha_2} /  {\alpha_1}},$   we have the representative $\langle\nabla_1+\nabla_2\rangle;$

\item if $\alpha_3\neq0,$ then choosing $y=0, z=-\frac{\alpha_2}{4\alpha_3\sqrt[5]{\alpha_3}},$ we have the family of representatives $\langle \lambda\nabla_1+\nabla_3\rangle.$
\end{enumerate}

Hence, we have the following 5-dimensional algebras:
\begin{longtable}{lllllllll}
 ${\mathcal N}^5_{07}(0)$ & : &
$e_1 e_1 = e_2$&   $e_1 e_2=e_3$&    $e_1 e_3=2e_4$ & $e_1e_4=3 e_5$ &&& \\
 ${\mathcal N}^5_{08}$ & : &
$e_1 e_1 = e_2$ &  $e_1 e_2=e_3$ & $e_1 e_3=e_4$& $e_1e_4=e_5$& $e_2e_1=e_5$&&\\
 ${\mathcal N}^5_{09}(\lambda)$ & : &
$e_1 e_1 = e_2$&   $e_1 e_2=e_3$ & $e_1 e_3=2e_4$ & $e_1e_4= \lambda e_5$ & $e_2e_3=2e_5$ & $e_3e_2=-2e_5$ & $e_4e_1=e_5$
\end{longtable}

\subsubsection{2-dimensional central extensions of ${\mathcal N}^4_{03}(0)$} 

Consider the vector space generated by the following two cocycles
$$\begin{array}{rcl}
\theta_1 &=& \alpha_1 \nabla_1+\alpha_2\nabla_2+\alpha_3\nabla_3,  \\
\theta_2 &=& \beta_1 \nabla_1+\beta_2\nabla_2.
\end{array}$$

If $\alpha_3=0,$ then we have the representative $\langle \nabla_1, \nabla_2 \rangle.$
If $\alpha_3\neq 0,$ then

\begin{enumerate}
\item if $\beta_1=0, \beta_2\neq 0,$ then we have the family of representatives $\langle \lambda \nabla_1+\nabla_3, \nabla_2 \rangle;$

\item if $\beta_1\neq0, \beta_2\neq 0,$ then choosing $x=\sqrt{ \beta_2 /  \beta_1},  y=0, z=\frac{x^3\alpha_1-x\alpha_2}{4\alpha_3},$ 
we have the  representative $\langle  \nabla_1+\nabla_2, \nabla_3 \rangle;$

\item if $\beta_1\neq0, \beta_2 = 0,$ then we have the representative $\langle  \nabla_1, \nabla_3 \rangle.$

\end{enumerate}

Hence we have the following 6-dimensional algebras:
\begin{longtable}{llllllllll}
 ${\mathcal N}^6_{05}(0)$ & : & $e_1 e_1 = e_2$ &
$e_1 e_2=e_3$ &    $e_1 e_3=2e_4$ & $e_1e_4=3
e_5$ &
$e_2 e_1= e_6$ &&&\\
 ${\mathcal N}^6_{06}(\lambda)$ & : &
$e_1 e_1 = e_2$&  $e_1 e_2=e_3$& $e_1 e_3=2e_4$& $e_1e_4=\lambda e_5$& $e_2e_1=e_6$ & $e_2e_3=2e_5$ & $e_3e_2=-2e_5$ & $e_4e_1=e_5$ \\
 ${\mathcal N}^6_{07}$ &: &
$e_1 e_1 = e_2$ &  $e_1 e_2=e_3$& $e_1 e_3=2e_4$& $e_1e_4= e_5$ & $e_2e_1=e_5$ & $e_2e_3=2e_6$ & $e_3e_2=-2e_6$& $e_4e_1=e_6$ \\

 ${\mathcal N}^6_{08}$ & : &
$e_1 e_1 = e_2$&   $e_1 e_2=e_3$& $e_1 e_3=2e_4$ & $e_1e_4= e_5$ & $e_2e_3=2e_6$ & $e_3e_2=-2e_6$ & $e_4e_1=e_6$

\end{longtable}

\subsection{Central extensions of ${\mathcal N}^4_{03}(1)$.} Let us use the following notations
\[\nabla_1=[\Delta_{21}], \quad \nabla_2=[\Delta_{14}]+[\Delta_{23}]+[\Delta_{32}]+[\Delta_{41}].\]

The automorphism group of ${\mathcal N}^4_{03}(1)$ consists of invertible matrices of the form
\[\phi=\begin{pmatrix}
x & 0 & 0  & 0\\
y & x^2 & 0 & 0\\
z & 2xy & x^3 & 0 \\
t & 2xz+y^2 & 3x^2y &  x^4\\
\end{pmatrix}. \]

Since
\[ \phi^T\begin{pmatrix}
0 & 0 & 0 &  \alpha_2\\
\alpha_1 & 0 & \alpha_2& 0  \\
0 & \alpha_2 & 0 & 0 \\
 \alpha_2 & 0 & 0 & 0
\end{pmatrix}\phi =
\begin{pmatrix}
\alpha^{***} & \alpha^{**} & \alpha^{*}& \alpha_2^* \\
\alpha_1^*+\alpha^{**} &\alpha^{*} &\alpha_2^*  & 0 \\
\alpha^{*} & \alpha_2^* & 0 & 0 \\
\alpha_2^*& 0 &0&0
\end{pmatrix},
\]
 we have that the action of ${\rm Aut} ({\mathcal N}^4_{03}(1))$ on the subspace
$\langle \sum\limits_{i=1}^2 \alpha_i\nabla_i  \rangle$
is given by
$\langle \sum\limits_{i=1}^2 \alpha_i^* \nabla_i  \rangle,$
where
\[\alpha_1^* =x^3\alpha_1, \quad \alpha_2^* =x^5\alpha_2. \]

\subsubsection{1-dimensional central extensions of ${\mathcal N}^4_{03}(1)$} 
We may assume $\alpha_2\neq0,$ so we have the following cases:
\begin{enumerate}
\item if $\alpha_1=0,$ then we have the representative $\langle\nabla_2\rangle;$
\item if $\alpha_1 \neq 0,$ then choosing
$x=\sqrt{{\alpha_1}  /{\alpha_2}}$   we have the representative $\langle\nabla_1+\nabla_2\rangle.$
\end{enumerate}

Hence, we have the following $5$-dimensional algebras:
\begin{longtable}{lllllll}
 ${\mathcal N}^5_{07}(1)$ & : &
$e_1 e_1 = e_2$&   $e_1 e_2=e_3$&    $e_1 e_3=e_4$ & $e_1e_4= e_5$ & $e_2 e_1=  e_3$\\
&& $e_2e_2=  e_4$ & $e_2e_3=e_5$ & $e_3e_1= e_4$ & $e_3e_2= e_5$ &  $e_4e_1= e_5$\\
${\mathcal N}^5_{10}$ & : &
$e_1 e_1 = e_2$&   $e_1 e_2=e_3$&    $e_1 e_3=e_4$ & $e_1e_4= e_5$ & $e_2 e_1= e_3+e_5$\\
&&  $e_2e_2= e_4$ & $e_2e_3=e_5$ & $e_3e_1=e_4$ & $e_3e_2=e_5$ & $e_4e_1=e_5$ 
\end{longtable}

\subsubsection{2-dimensional central extensions of ${\mathcal N}^4_{03}(1)$} 
Since $\dim\left({\rm H^2}( {\mathcal N}^4_{03}(1))\right)=2,$ there is only one 2-dimensional central extension of  ${\mathcal N}^4_{03}(1).$
Hence, we have the following $6$-dimensional algebra:
\begin{longtable}{lllllll}
 ${\mathcal N}^6_{05}(1)$ & : & $e_1 e_1 = e_2$ &
$e_1 e_2=e_3$ &    $e_1 e_3=e_4$ & $e_1e_4=
e_5$ &
$e_2 e_1=  e_3+e_6$  \\
&&$e_2e_2= e_4$ & $e_2e_3=e_5$ & $e_3e_1= e_4$ & $e_3e_2=e_5$ &  $e_4e_1= e_5$ 
\end{longtable}

\subsection{Central extensions of ${\mathcal N}^4_{04}$}Let us use the following notations
\[\nabla_1=[\Delta_{13}], \quad \nabla_2=2[\Delta_{14}]+[\Delta_{22}]+[\Delta_{31}],  \quad \nabla_3=[\Delta_{14}]-2[\Delta_{23}]+2[\Delta_{32}]-[\Delta_{41}].\]

The automorphism group of ${\mathcal N}^4_{04}$ consists of invertible matrices of the form
\[\phi=\begin{pmatrix}
1 & 0 & 0  & 0\\
x & 1 & 0 & 0\\
y & x & 1 & 0 \\
z & x+2y & 2x &  1\\
\end{pmatrix}. \]

Since
\[ \phi^T\begin{pmatrix}
0 &0 &  \alpha_1 &  2\alpha_2+\alpha_3\\
0 & \alpha_2 &  -2\alpha_3& 0  \\
\alpha_2 & 2\alpha_3 & 0 & 0 \\
 -\alpha_3 & 0 & 0 & 0
\end{pmatrix}\phi =
\begin{pmatrix}
\alpha^{***} &\alpha^{**} &  2\alpha^*+\alpha_1^*  &  2\alpha_2^*+\alpha_3^*\\
\alpha^* & \alpha_2^* &  -2\alpha_3^*& 0  \\
\alpha_2^* & 2\alpha_3^* & 0 & 0 \\
 -\alpha_3^* & 0 & 0 & 0
\end{pmatrix},
\]
 we have that the action of ${\rm Aut} ({\mathcal N}^4_{04})$ on the subspace
$\langle \sum\limits_{i=1}^3 \alpha_i\nabla_i  \rangle$
is given by
$\langle \sum\limits_{i=1}^3 \alpha_i^* \nabla_i \rangle,$
where
\[\alpha_1^*=\alpha_1+2(x-2x^2+4y)\alpha_3, \quad \alpha_2^* =\alpha_2, \quad \alpha_3^* =\alpha_3. \]

\subsubsection{1-dimensional central extensions of ${\mathcal N}^4_{04}$}
Since we are interested only in new algebras, we assume  $(\alpha_2, \alpha_3) \neq (0,0)$ and we have the following cases:
\begin{enumerate}
\item if $\alpha_2\neq0, \alpha_3=0,$ then we have the family of representatives $\langle \lambda\nabla_1+\nabla_2\rangle;$

\item if $\alpha_3\neq0,$ then by choosing $ y=-\frac{\alpha_1+(2x-4x^2)\alpha_3}{8\alpha_3},$ we have the family of representatives $\langle\lambda\nabla_2+\nabla_3\rangle.$
\end{enumerate}

Hence, we have  the following 5-dimensional algebras:
\begin{longtable}{lllllllll}
 ${\mathcal N}^5_{11}(\lambda)$ & : &
$e_1 e_1 = e_2$ &   $e_1 e_2=e_3$ &    $e_1 e_3=2e_4 + \lambda e_5$ &  $e_1e_4=2e_5$ & $e_2 e_1= e_4$ &  $e_2e_2=e_5$ & $e_3e_1=e_5$ \\
 ${\mathcal N}^5_{12}(\lambda)$ & : &
$e_1 e_1 = e_2$&   $e_1 e_2=e_3$&    $e_1 e_3=2e_4$ & $e_1e_4=(2\lambda+1) e_5$ & $e_2 e_1= e_4$&&\\
&&  $e_2e_2=\lambda e_5$ & $e_2e_3=-2e_5$ & $e_3e_1=\lambda e_5$ & $e_3e_2=2e_5$ & $e_4e_1=-e_5$&&
\end{longtable}

\subsubsection{2-dimensional central extensions of ${\mathcal N}^4_{04}$}

Consider the vector space generated by the following two cocycles
$$\begin{array}{rcl}
\theta_1 &=& \alpha_1 \nabla_1+\alpha_2\nabla_2+\alpha_3\nabla_3,  \\
\theta_2 &=& \beta_1 \nabla_1+\beta_2\nabla_2.
\end{array}$$

If $\alpha_3=0,$ then we have the representative $\langle \nabla_1, \nabla_2 \rangle.$
If $\alpha_3\neq 0,$ then
\begin{enumerate}
    \item if $\beta_2=0,$ then we have the family of representatives
    $\langle \nabla_1, \lambda \nabla_2+\nabla_3 \rangle;$

    \item if $\beta_2\neq 0,$  then we have the family of representatives
    $\langle \lambda \nabla_1+\nabla_2, \nabla_3 \rangle.$
\end{enumerate}

Hence, we have the following 6-dimensional algebras:
\begin{longtable}{lllllllll}
 ${\mathcal N}^6_{09}$ & : &
$e_1 e_1 = e_2$ &   $e_1 e_2=e_3$ &    $e_1 e_3=2e_4+e_6$ & $e_1e_4=2e_5$ & $e_2 e_1= e_4$ & $e_2e_2=e_5$& $e_3e_1=e_5$  \\

 ${\mathcal N}^6_{10}(\lambda)$ & : &
$e_1 e_1 = e_2$&   $e_1 e_2=e_3$&    $e_1 e_3=2e_4+e_6$ & $e_1e_4=(2\lambda+1) e_5$ & $e_2 e_1= e_4$ &&\\
&& $e_2e_2=\lambda e_5$ & $e_2e_3=-2e_5$  &   $e_3e_1=\lambda e_5$ & $e_3e_2=2e_5$ & $e_4e_1=-e_5$  &&\\

 ${\mathcal N}^6_{11}(\lambda)$ & : &
$e_1 e_1 = e_2$   & $e_1 e_2=e_3$&    $e_1 e_3=2e_4+\lambda e_6$ & $e_1e_4=e_5+ 2e_6$ & $e_2 e_1= e_4$ && \\
&&$e_2e_2=e_6$& $e_2e_3=-2e_5$  & $e_3e_1=e_6$ & $e_3e_2=2e_5$ & $e_4e_1=-e_5$ && 

\end{longtable}

\subsection{Central extensions of ${\mathcal N}^4_{05}$} Let us use the following notations
\[\nabla_1=[\Delta_{21}], \quad \nabla_2=-2[\Delta_{13}]+[\Delta_{14}]+[\Delta_{23}]+[\Delta_{32}]+[\Delta_{41}].\]

The automorphism group of ${\mathcal N}^4_{05}$ consists of invertible matrices of the form
\[\phi=\begin{pmatrix}
1 & 0 & 0  & 0\\
x& 1 & 0 & 0\\
y & 2x & 1 & 0 \\
z & x^2+x+2y & 3x &  1\\
\end{pmatrix}. \]

Since
\[ \phi^T\begin{pmatrix}
0 &0 &  -2\alpha_2 &  \alpha_2\\
\alpha_1 & 0 &  \alpha_2& 0  \\
0 & \alpha_2 & 0 & 0 \\
 \alpha_2 & 0 & 0 & 0
\end{pmatrix}\phi =
\begin{pmatrix}
\alpha^{***} &\alpha^{**} & -2\alpha_2^{*}+ \alpha^{*} &  \alpha_2^{*}\\
\alpha_1^{*}+\alpha^{*}+\alpha^{**} & \alpha^{*} &  \alpha_2^{*}& 0  \\
\alpha^{*} & \alpha_2^{*} & 0 & 0 \\
 \alpha_2^{*} & 0 & 0 & 0
\end{pmatrix},
\]
 we have that the action of ${\rm Aut} ({\mathcal N}^4_{05})$ on the subspace
$\langle \sum\limits_{i=1}^2\alpha_i\nabla_i  \rangle$
is given by
$\langle \sum\limits_{i=1}^2\alpha_i^{*}\nabla_i\rangle,$
where
\[\alpha_1^{*}=\alpha_1, \quad \alpha_2^{*}=\alpha_2. \]

\subsubsection{1-dimensional central extensions of ${\mathcal N}^4_{05}$}
Assuming  $\alpha_2\neq0,$ we have the family of representatives $\langle \lambda\nabla_1+\nabla_2\rangle.$

Hence, we have the following $5$-dimensional algebras:
\begin{longtable}{lllllll}
 ${\mathcal N}^5_{13}(\lambda)$ & :&
$e_1 e_1 = e_2$ & $e_1 e_2=e_3$&  $e_1 e_3=e_4-2e_5$ & $e_1e_4=e_5$ & $e_2 e_1= e_3 + e_4+\lambda e_5$\\
&&  $e_2 e_2= e_4$ & $e_2e_3=e_5$ & $e_3 e_1= e_4$ & $e_3e_2=e_5$ & $e_4e_1=e_5$
\end{longtable}

\subsubsection{2-dimensional central extensions of ${\mathcal N}^4_{05}$}
Since $\dim\left({\rm H^2}( {\mathcal N}^4_{05})\right)=2,$ there is only one 2-dimensional central extension of  ${\mathcal N}^4_{05}.$ Hence, we have the following $6$-dimensional algebra:
\begin{longtable}{lllllll}
 ${\mathcal N}^6_{12}$ & :&
$e_1 e_1 = e_2$ & $e_1 e_2=e_3$&  $e_1 e_3=e_4-2e_5$ & $e_1e_4=e_5$& $e_2 e_1= e_3 + e_4+e_6$  \\
&&  $e_2 e_2= e_4$&  $e_2e_3=e_5$ & $e_3 e_1= e_4$ & $e_3e_2=e_5$& $e_4e_1=e_5$
\end{longtable}

\

Summarizing the results of the present sections
and using the information from Table A from Appendix, we have the following theorem.

\begin{theorem}
Let $\mathcal N$ be a $5$-dimensional complex one-generated nilpotent Novikov algebra.
Then $\mathcal N$ is isomorphic to one of the algebras  ${\mathcal N}_j^5,$ $j=01, \ldots, 13$ listed in Table B presented in Appendix.
\end{theorem}

\section{Classification of $6$-dimensional one-generated nilpotent Novikov algebras}

The multiplication tables of $5$-dimenensional one-generated nilpotent Novikov algebras are given in Table B presented in Appendix.
All the necessary information about
coboundaries, cocycles and  second cohomology spaces of $5$-dimensional one-generated nilpotent Novikov algebras is given in Table C presented in Appendix.

\subsection{Central extensions of ${\mathcal N}^5_{01}$}

Let us use the following notations
\[\nabla_1=[\Delta_{14}], \quad \nabla_2=2[\Delta_{13}]+[\Delta_{22}]+[\Delta_{41}], \quad  \nabla_3=2[\Delta_{15}]-[\Delta_{23}]-[\Delta_{51}].\]

The automorphism group of ${\mathcal N}^5_{01}$ consists of invertible matrices of the form
\[\phi=\begin{pmatrix}
x & 0 & 0  & 0& 0\\
y & x^2 & 0 & 0& 0\\
z & xy & x^3 & 0& 0 \\
v & xy & 0 &  x^3& 0\\
w & 0 & -x^2y &  x^2y& x^4\\
\end{pmatrix}. \]

Since
\[ \phi^T\begin{pmatrix}
0 &0 &  2\alpha_2 &  \alpha_1& 2\alpha_3\\
0 & \alpha_2 &  -\alpha_3& 0 & 0 \\
0 & 0 & 0 & 0 & 0\\
 \alpha_2 & 0 & 0 & 0 & 0\\
-\alpha_3 & 0 & 0 & 0 & 0
\end{pmatrix}\phi =
\begin{pmatrix}
\alpha^{****} & \alpha^{***} & 2\alpha_2^*+\alpha^* &  \alpha_1^* &2\alpha_3^*\\
\alpha^{**} & \alpha_2^* &  -\alpha_3^*& 0 &0 \\
-\alpha^* & 0 & 0 & 0 &0\\
\alpha_2^* & 0 & 0 & 0 &0\\
  -\alpha_3^* & 0 & 0 & 0 &0\\
\end{pmatrix},
\]
 we have that the action of ${\rm Aut} ({\mathcal N}^5_{01})$ on the subspace
$\langle \sum\limits_{i=1}^3 \alpha_i\nabla_i  \rangle$
is given by
$\langle \sum\limits_{i=1}^3 \alpha_i^* \nabla_i \rangle,$
where
\[\alpha_1^*=x^4\alpha_1+2x^3y\alpha_3, \quad \alpha_2^* =x^4\alpha_2-x^3y\alpha_3, \quad \alpha_3^* =x^5\alpha_3. \]

We assume $ \alpha_3\neq0$, since otherwise we obtain algebras with $2$-dimensional annihilator.
Then we have the following cases:
\begin{enumerate}
\item if $\alpha_1=-2\alpha_2,$ then choosing $y= \frac {\alpha_2x} {\alpha_3},$
we have the representative $\langle \nabla_3\rangle,$ which gives an algebra with $2$-dimensional annihilator;

\item if $\alpha_1\neq-2\alpha_2,$ then choosing $y= \frac {\alpha_2x} {\alpha_3}, x=\frac{\alpha_1+2\alpha_2}{\alpha_3}$, we have the representative $\langle \nabla_1+\nabla_3\rangle.$
\end{enumerate}

Hence, we have only one new algebra: 
\begin{longtable}{ll lllll}
${\mathcal N}^6_{13}$ &:& $e_1 e_1 = e_2$  & $e_1e_2=e_4$ & $e_1e_3=e_5$ & $e_1e_4=e_6$ &$e_1e_5=2e_6$  \\
&& $e_2 e_1=e_3$ & $e_2e_3=-e_6$ & $e_3e_1=-e_5$ & $e_5e_1=-e_6$ &
\end{longtable}

\subsection{Central extensions of ${\mathcal N}^5_{02}(\lambda)_{\lambda\neq0}$} Let us use the following notations
\[\nabla_1=[\Delta_{14}]-[\Delta_{41}], \quad \nabla_2=[\Delta_{13}]-\lambda[\Delta_{41}], \quad \nabla_3=(3-2\lambda)[\Delta_{15}]+\lambda(2-\lambda)[\Delta_{23}]+\lambda[\Delta_{32}]+\lambda[\Delta_{51}] .\]

The automorphism group of ${\mathcal N}^5_{02}(\lambda)$ consists of invertible matrices of the form
\[\phi=\begin{pmatrix}
x & 0 & 0  & 0& 0\\
y & x^2 & 0 & 0& 0\\
z & (\lambda+1)xy & x^3 & 0& 0 \\
v & xy & 0 &  x^3& 0\\
w & \lambda y^2+2xz & (-\lambda^2+2\lambda+2)x^2y &  \lambda^2(\lambda-1)x^2y& x^4\\
\end{pmatrix}. \]

Since
\[ \phi^T\begin{pmatrix}
0 &0 &\alpha_2 & \alpha_1& (3-2\lambda)\alpha_3\\
0 & 0 &  \lambda(2-\lambda)\alpha_3& 0 & 0 \\
0 & \lambda\alpha_3 & 0 & 0 & 0\\
 -\alpha_1-\lambda\alpha_2 & 0 & 0 & 0 & 0\\
\lambda\alpha_3 & 0 & 0 & 0 & 0
\end{pmatrix}\phi =
\begin{pmatrix}
\alpha^{****} & \alpha^{***} & \alpha_2^*+(2-\lambda)\alpha^{*} & \alpha_1^*& (3-2\lambda)\alpha_3^*\\
\alpha^{**} & \lambda \alpha^{*} &  \lambda(2-\lambda)\alpha_3^*& 0 & 0 \\
\lambda \alpha^{*} & \lambda\alpha_3^* & 0 & 0 & 0\\
- \alpha_1^*-\lambda\alpha_2^* & 0 & 0 & 0 & 0\\
\lambda\alpha_3^* & 0 & 0 & 0 & 0
\end{pmatrix},
\]
 we have that the action of ${\rm Aut} ({\mathcal N}^5_{02}(\lambda)_{\lambda\neq0})$ on the subspace $\langle \sum\limits_{i=1}^3 \alpha_i\nabla_i  \rangle$
is given by
$\langle \sum\limits_{i=1}^3 \alpha_i^* \nabla_i \rangle,$
where
\[\alpha_1^*=x^4\alpha_1+x^3y\lambda^2(\lambda-1)(3-2\lambda)\alpha_3, \quad \alpha_2^* =x^4\alpha_2+x^3y\lambda(\lambda-1)(\lambda-3)\alpha_3, \quad \alpha_3^* =x^5\alpha_3. \]

We assume $ \alpha_3\neq0$ and $(\alpha_1, \alpha_2) \neq (0, 0).$ Then we have the following cases
\begin{enumerate}
\item $\lambda=1.$ Then
\begin{enumerate}
\item $\alpha_2 \neq 0,$ then choosing $x=\frac{\alpha_2} {\alpha_3},$ we have the family of representatives $\langle \alpha\nabla_1+\nabla_2+\nabla_3\rangle;$
\item $\alpha_2 = 0,$ then choosing
$x= \frac{\alpha_1} {\alpha_3}$, we have the representative $\langle \nabla_1+\nabla_3\rangle.$
\end{enumerate}

\item $\lambda= 3.$ If $\alpha_2\neq 0,$ then choosing $x= \frac{\alpha_2} {\alpha_3},$ $y= \frac {x\alpha_1}{54\alpha_3}$
we have the representative $\langle \nabla_2+\nabla_3\rangle.$ If
$\alpha_2 = 0,$ then we have an algebra with $2$-dimensional annihilator.

\item  $\lambda\not \in \{ 1,  3\}.$ Then in the case of $\alpha_1 \neq \frac{\alpha_2\lambda(3-2\lambda)}{\lambda-3},$ choosing
$x=\frac{\alpha_1(\lambda-3)-\alpha_2\lambda(3-2\lambda)}{\alpha_3(\lambda-3)},$ $y=\frac {-x\alpha_2}{\alpha_3\lambda(\lambda-1)(\lambda-3)},$ 
we have the representative $\langle \nabla_1+\nabla_3\rangle.$ If
$\alpha_1 = \frac{\alpha_2\lambda(3-2\lambda)}{\lambda-3},$ then we have an algebra with $2$-dimensional annihilator.
\end{enumerate}

Thus, we have three representatives of distinct orbits
$\langle\alpha\nabla_1+\nabla_2+\nabla_3\rangle_{\lambda=1},  \langle\nabla_2+ \nabla_3\rangle_{\lambda=3}, 
\langle\nabla_1+ \nabla_3\rangle_{(\lambda \notin\{0; 3\})},$
which give the following algebras:

\begin{longtable}{ll lllll} ${\mathcal N}^6_{14}(\alpha)$ &:&
$e_1 e_1 = e_2$  & $e_1e_2=e_3$ & $e_1e_3=e_5+e_6$ & $e_1e_4=\alpha e_6$ &\\
&& $e_1e_5=e_6$ & $e_2 e_1=e_3+e_4$ & $e_2e_2=e_5$ & $e_2e_3=e_6$ &\\
&& $e_3e_1=e_5$ & $e_3e_2=e_6$  & $e_4e_1=-(\alpha+1)e_6$ & $e_5e_1=e_6$ & \\
 ${\mathcal N}^6_{15}$ &:&
$e_1 e_1 = e_2$  & $e_1e_2=e_3$ & $e_1e_3=-e_5 + e_6$ & $e_1e_5=-3 e_6$ &\\
&& $e_2 e_1= 3 e_3+e_4$  & $e_2e_2= 3 e_5$ & $e_2e_3=-3 e_6$ &  $e_3e_1=3 e_5$ &\\
&& $e_3e_2=3 e_6$  & $e_4e_1=-3 e_6$ & $e_5e_1=3 e_6$ &&\\
 ${\mathcal N}^6_{16}(\lambda)_{\lambda \notin\{0, 3\}}$ &:&
$e_1 e_1 = e_2$  & $e_1e_2=e_3$ & $e_1e_3=(2-\lambda)e_5 $ & $e_1e_4= e_6 $ &\\
& & $e_1e_5=(3-2\lambda)e_6$ & $e_2 e_1=\lambda e_3+e_4$  & $e_2e_2=\lambda e_5$ & $e_2e_3=\lambda(2-\lambda) e_6$ &\\
& & $e_3e_1=\lambda e_5$ & $e_3e_2=\lambda e_6$  & $e_4e_1=- e_6$ & $e_5e_1=\lambda e_6$ &
\end{longtable}

\subsection{Central extensions of ${\mathcal N}^5_{02}(0)$} Let us use the following notations
\[\nabla_1=[\Delta_{14}]-[\Delta_{41}], \quad \nabla_2=2[\Delta_{14}]+[\Delta_{22}]+[\Delta_{31}],  \quad \nabla_3=[\Delta_{15}], \quad \nabla_4=2[\Delta_{23}]-2[\Delta_{32}]+[\Delta_{51}].\]

The automorphism group of ${\mathcal N}^5_{02}(0)$ consists of invertible matrices of the form
\[\phi=\begin{pmatrix}
x & 0 & 0  & 0& 0\\
y & x^2 & 0 & 0& 0\\
z & xy & x^3 & 0& 0 \\
v & xy & 0 &  x^3& 0\\
w & 2xz & 2x^2y &  0& x^4\\
\end{pmatrix}. \]

Since
\[ \phi^T\begin{pmatrix}
0 &0 & 0 & \alpha_1+2\alpha_2& \alpha_3\\
0 & \alpha_2 &  2\alpha_4& 0 & 0 \\
\alpha_2 & -2\alpha_4 & 0 & 0 & 0\\
 -\alpha_1 & 0 & 0 & 0 & 0\\
\alpha_4 & 0 & 0 & 0 & 0
\end{pmatrix}\phi =
\begin{pmatrix}
\alpha^{****} & \alpha^{***} & \alpha^{**} & \alpha_1^*+2\alpha_2^*& \alpha_3^*\\
\alpha^{*} & \alpha_2^* &  2\alpha_4^*& 0 & 0 \\
\alpha_2^* & -2\alpha_4^* & 0 & 0 & 0\\
 -\alpha_1^* & 0 & 0 & 0 & 0\\
\alpha_4^* & 0 & 0 & 0 & 0
\end{pmatrix},
\]
 we have that the action of ${\rm Aut} ({\mathcal N}^5_{02}(0))$ on the subspace
$\langle \sum\limits_{i=1}^4 \alpha_i\nabla_i  \rangle$
is given by
$\langle \sum\limits_{i=1}^4 \alpha_i^* \nabla_i \rangle,$
where
\[\alpha_1^* =x^4\alpha_1, \quad \alpha_2^* =x^4\alpha_2, \quad \alpha_3^*=x^5\alpha_3, \quad \alpha_4^* =x^5\alpha_4. \]

We assume  $ (\alpha_1,\alpha_2)\neq(0,0)$ and $(\alpha_3,\alpha_4)\neq(0,0).$
Then we have the following cases:
\begin{enumerate}
\item if $\alpha_2=0, \alpha_4=0,$ then $\alpha_1\neq0$, $\alpha_3\neq 0$ and choosing $x=\frac{\alpha_1}{\alpha_3},$ we have the representative $\langle\nabla_1+\nabla_3\rangle;$

\item if $\alpha_2=0, \alpha_4\neq0,$ then $\alpha_1\neq0$ and choosing $x= \frac{\alpha_1} {\alpha_4},$ we have the family of representatives $\langle \nabla_1+\alpha\nabla_3+\nabla_4\rangle;$

\item if $\alpha_2\neq0, \alpha_4=0,$ then $\alpha_3\neq0$ and choosing $x=\frac{\alpha_2}{\alpha_3},$ we have the family of representatives $\langle\alpha \nabla_1+\nabla_2 +\nabla_3 \rangle;$

\item if $\alpha_2\neq0, \alpha_4\neq0,$ then choosing
$x= \frac{\alpha_2} {\alpha_4},$ we have the family of representatives $\langle \alpha\nabla_1+\nabla_2+\beta\nabla_3+\nabla_4\rangle.$

\end{enumerate}

Hence, we have the following  new algebras:
\begin{longtable}{ll lllllll}
${\mathcal N}^6_{16}(0)$ &:&
$e_1 e_1 = e_2$  & $e_1e_2=e_3$ & $e_1e_3=2e_5 $ & $e_1e_4= e_6 $ &&&\\
& & $e_1e_5=3e_6$ & $e_2 e_1=e_4$  & $e_4e_1=- e_6$ &  &\\

 ${\mathcal N}^6_{17}(\alpha)$ &:&
$e_1 e_1 = e_2$  & $e_1e_2=e_3$ & $e_1e_3=2e_5$ & $e_1e_4= e_6$ &  $e_1e_5=\alpha e_6$ & $e_2 e_1=e_4$ \\
&&  $e_2e_3=2e_6$ &  $e_3e_1=e_6$ & $e_3e_2=-2 e_6$  & $e_4e_1=-e_6$ & $e_5e_1= e_6$ \\

 ${\mathcal N}^6_{18}(\alpha)$ &:&
$e_1 e_1 = e_2$  & $e_1e_2=e_3$ & $e_1e_3=2e_5$ & $e_1e_4= (\alpha+2)e_6$ & $e_1e_5=e_6$ \\
&& $e_2 e_1=e_4$ & $e_2e_2=e_6$ &  $e_3e_1=e_6$ & $e_4e_1=-\alpha e_6$ &
\\
 ${\mathcal N}^6_{19}(\alpha, \beta)$ &:&
$e_1 e_1 = e_2$  & $e_1e_2=e_3$ & $e_1e_3=2e_5$ & $e_1e_4= (\alpha+2)e_6$ & $e_1e_5=\beta e_6$ & $e_2 e_1=e_4$ \\
&& $e_2e_2=e_6$ & $e_2e_3=2e_6$ & $e_3e_1=e_6$ & $e_3e_2=-2 e_6$  & $e_4e_1=-\alpha e_6$ & $e_5e_1= e_6$ &
\end{longtable}

\subsection{Central extensions of ${\mathcal N}^5_{03}(\lambda, \mu)$.}
Let us use the following notations
\[\nabla_1=
[\Delta_{13}], \quad  \nabla_2=
[\Delta_{14}] - [\Delta_{41}], \]  
\[\nabla_3=(2\mu-1)[\Delta_{15}]+\lambda(2-\mu) [\Delta_{23}]+\mu(2-\mu) [\Delta_{24}] +\Big(3 -(\lambda+1)(2-\mu) \Big)[\Delta_{32}]+(2-\mu)[\Delta_{42}]+(2-\mu)[\Delta_{51}].\]

 The automorphism group of ${\mathcal N}^5_{03}(\lambda, \mu)$ consists of invertible matrices of the form
\[\phi=\begin{pmatrix}
               x & 0 & 0 & 0 & 0 \\
               y & x^2 & 0 & 0 & 0 \\
               z & xy & x^3 & 0 & 0 \\
               v & xy & 0 & x^3 & 0 \\
               w & (\lambda+1)xz+2xv+y^2 & (\lambda+\mu+1)x^2y & (4-\mu)x^2y & x^4
             \end{pmatrix}. \]

 Since
\[ \phi^T\begin{pmatrix}
0 & 0 & \alpha_{1} & \alpha_{2} & (2\mu-1) \alpha_{3} \\
               0 & 0 & \lambda(2-\mu) \alpha_{3} & \mu(2-\mu) \alpha_{3} & 0 \\
               0 & \Big(3 - (\lambda+1)(2-\mu) \Big) \alpha_{3} & 0 & 0 & 0\\
               -\alpha_{2} & (2-\mu) \alpha_{3} & 0 & 0 & 0 \\
               (2-\mu)\alpha_{3} & 0 & 0 & 0 & 0
\end{pmatrix}\phi =\]\[
\begin{pmatrix}
\alpha^{****} & \alpha^{***} & \alpha_{1}^*+ \lambda \alpha^{*} & \alpha_{2}^*+ \mu \alpha^{*} & (2\mu-1)\alpha_{3}^* \\
               \alpha^{**} & \alpha^{*} & \lambda(2-\mu)\alpha_{3}^* & \mu(2-\mu)\alpha_{3}^* & 0 \\
               \alpha^{*} & \Big(3 - (\lambda+1)(2-\mu) \Big) \alpha_{3}^* & 0 & 0 & 0\\
               (2-\mu)\alpha^{*}-\alpha_{2}^* & (2-\mu) \alpha_{3}^* & 0 & 0 & 0 \\
               (2-\mu) \alpha_{3}^* & 0 & 0 & 0 & 0
\end{pmatrix},
\]
we have that the action of ${\rm Aut} ({\mathcal N}^5_{03}(\lambda, \mu))$ on the subspace
$\langle \sum\limits_{i=1}^3 \alpha_i\nabla_i  \rangle$
is given by
$\langle \sum\limits_{i=1}^3 \alpha_i^* \nabla_i \rangle,$
where
\[\alpha_1^* = x^4 \alpha_1 + x^3y \alpha_3 (\mu+1)\Big(2\mu-1 +\lambda(\mu-2) \Big), \quad
\alpha_2^* = x^4 \alpha_2 + x^3y\alpha_3 (\mu-1)(\mu-2)^2, \quad
\alpha_3^*=x^5 \alpha_3.\]

Assuming $\alpha_3\neq 0,$  we have the following cases:
\begin{enumerate}

\item if  $\mu \notin\{1, 2\},$ then choosing $y=-\frac{x\alpha_2}{(\mu-1)(\mu-2)^2\alpha_3}$, we get $\alpha_2^*=0$ and
have the representatives $\langle\nabla_3\rangle$ and $\langle\nabla_1 + \nabla_3\rangle,$ depending on whether
$\alpha_1 - \frac{(\mu+1)(2\mu-1+\lambda(\mu-2))}{(\mu-1)(\mu-2)^2}\alpha_2 =0$ or not;

\item if  $\mu = 1,$ then we consider following subcases
    \begin{enumerate}

    \item if  $\lambda \neq 1,$ then  choosing $y=\frac{x\alpha_1}{2(\lambda-1)\alpha_3}$, we 
    have the representatives $\langle\nabla_3\rangle$ and $\langle\nabla_2 + \nabla_3\rangle,$
    depending on whether $\alpha_2=0$ or not;
    \item if  $\lambda = 1,$ then we get $\alpha_1^* = x^4 \alpha_1, \
\alpha_2^* = x^4 \alpha_2, \ \alpha_3^*=x^5 \alpha_3$ and have the representatives $\langle\nabla_3\rangle,$  $\langle\nabla_1 + \nabla_3\rangle$ and $\langle  \nu \nabla_1 +  \nabla_2 + \nabla_3\rangle$;

    \end{enumerate}
    
\item if  $\mu = 2,$  then choosing $y=-\frac{x\alpha_1}{9\alpha_3}$, we 
    have the representatives $\langle\nabla_3\rangle$ and $\langle\nabla_2 + \nabla_3\rangle,$
    depending on whether $\alpha_2=0$ or not.

\end{enumerate}

Thus, we obtain the following representatives of distinct orbits
\[\langle\nabla_3\rangle,  \quad \langle\nabla_1+ \nabla_3\rangle_{\mu\neq2}, \quad
\langle\nabla_2+ \nabla_3\rangle_{\mu =1}, \quad \langle  \nu \nabla_1 +  \nabla_2 + \nabla_3\rangle_{\lambda=1, \mu=1, \nu\neq 0}, \quad \langle\nabla_2+ \nabla_3\rangle_{\mu =2},\]

Hence, we have  the  following new algebras:
\begin{longtable}{ll lllll}
${\mathcal N}^6_{20}(\lambda, \mu)$ & : &
$e_1 e_1 = e_2$ &    $e_1e_2=e_3$ & $e_1 e_3=\lambda e_5$ & $e_1e_4=\mu e_5$ & $e_1e_5=(2\mu-1) e_6$\\
&& $e_2 e_1=e_4$ & $e_2 e_2=e_5$ &  $e_2 e_3=\lambda (2-\mu)e_6$ &  $e_2 e_4=\mu (2-\mu) e_6$ & \\
&& $e_3 e_1= e_5$&  $e_4e_1=(2-\mu)e_5$ & $e_3 e_2=\Big(3-(\lambda+1)(2-\mu) \Big)e_6$ & $e_4 e_2= (2-\mu)e_6$ & $e_5 e_1=(2-\mu)e_6$ \\
${\mathcal N}^6_{21}(\lambda, \mu)_{\mu\neq2}$ & : &
$e_1 e_1 = e_2$ &    $e_1e_2=e_3$ & $e_1 e_3=\lambda e_5 +e_6$ & $e_1e_4=\mu e_5$ & $e_1e_5=(2\mu-1) e_6$\\
&& $e_2 e_1=e_4$ & $e_2 e_2=e_5$ &  $e_2 e_3=\lambda(2-\mu) e_6$ &  $e_2 e_4=\mu(2-\mu) e_6$ & \\
&& $e_3 e_1= e_5$ &  $e_4e_1=(2-\mu)e_5$ & $e_3 e_2=\Big(3-(\lambda+1)(2-\mu)\Big)e_6$ & $e_4 e_2= (2-\mu) e_6$ & $e_5 e_1= (2-\mu)e_6$ \\
${\mathcal N}^6_{22}(\lambda)$ & : &
$e_1 e_1 = e_2$ &    $e_1e_2=e_3$ & $e_1 e_3=\lambda e_5$ & $e_1e_4=e_5+e_6$ & $e_1e_5= e_6$\\
&& $e_2 e_1=e_4$ & $e_2 e_2=e_5$ &  $e_2 e_3=\lambda e_6$ &  $e_2 e_4=  e_6$ & \\
&& $e_3 e_1= e_5$&  $e_4e_1= e_5-e_6$ & $e_3 e_2=(2-\lambda)e_6$ & $e_4 e_2= e_6$ & $e_5 e_1=e_6$ \\
${\mathcal N}^6_{23}(\nu)_{\nu\neq0}$ & : &
$e_1 e_1 = e_2$ &    $e_1e_2=e_3$ & $e_1 e_3= e_5 +  \nu e_6$ & $e_1e_4=e_5+e_6$ & $e_1e_5= e_6$\\
&& $e_2 e_1=e_4$ & $e_2 e_2=e_5$ &  $e_2 e_3= e_6$ &  $e_2 e_4=  e_6$ & \\
&& $e_3 e_1= e_5$&  $e_4e_1= e_5-e_6$ & $e_3 e_2=e_6$ & $e_4 e_2= e_6$ & $e_5 e_1=e_6$ \\

${\mathcal N}^6_{24}(\lambda)$ & : &
$e_1 e_1 = e_2$ &    $e_1e_2=e_3$ & $e_1 e_3=\lambda e_5$ & $e_1e_4=2 e_5 + e_6$ & $e_1e_5=e_6$\\
&& $e_2 e_1=e_4$ & $e_2 e_2=e_5$ & $e_3 e_1= e_5$&  $e_3e_2=e_6$ &
\end{longtable}

\relax
\subsection{Central extensions of ${\mathcal N}^5_{04}(\lambda)_{\lambda \neq 0}$}

Let us use the following notations
  \[\nabla_1=
[\Delta_{14}] - [\Delta_{41}], \quad \nabla_2=
[\Delta_{22}]+[\Delta_{31}]+2[\Delta_{41}], \quad \nabla_3=-2[\Delta_{15}]+[\Delta_{23}]+\lambda[\Delta_{24}]-[\Delta_{32}]+[\Delta_{51}].
 \]

 The automorphism group of ${\mathcal N}^5_{04}(\lambda)$ consists of invertible matrices of the form
\[\phi=\begin{pmatrix}
               x & 0 & 0 & 0 & 0 \\
               y & x^2 & 0 & 0 & 0 \\
               z & xy & x^3 & 0 & 0 \\
               v & xy & 0 & x^3 & 0 \\
               w & xz & (\lambda+1)x^2y & -\lambda x^2y & x^4
             \end{pmatrix}. \]
 Since
\[ \phi^T\begin{pmatrix}
0 & 0 & 0 & \alpha_{1} & -2\alpha_{3} \\
               0 & \alpha_{2} & \alpha_{3} & \lambda\alpha_{3} & 0 \\
               \alpha_{2} & -\alpha_{3} & 0 & 0 & 0\\
               2\alpha_{2}-\alpha_{1} & 0 & 0 & 0 & 0 \\
               \alpha_{3} & 0 & 0 & 0 & 0
\end{pmatrix}\phi =
\begin{pmatrix}
\alpha^{****} & \alpha^{***} & \alpha^{*} & \alpha_{1}^*+\lambda \alpha^{*} & -2\alpha_{3}^* \\
               \alpha^{**} & \alpha_{2}^* & \alpha_{3}^* & \lambda\alpha_{3}^* & 0 \\
               \alpha_{2}^* & -\alpha_{3}^* & 0 & 0 & 0\\
               2\alpha_{2}^*-\alpha_{1}^*-\lambda\alpha^{*} & 0 & 0 & 0 & 0 \\
               \alpha_{3}^* & 0 & 0 & 0 & 0
\end{pmatrix},
\]
 we have that the action of ${\rm Aut} ({\mathcal N}^5_{04}(\lambda))$ on the subspace
$\langle \sum\limits_{i=1}^3 \alpha_i\nabla_i  \rangle$
is given by
$\langle \sum\limits_{i=1}^3 \alpha_i^* \nabla_i \rangle,$
where
\[\alpha_1^* = x^4 \alpha_1 + 2x^3y \lambda(\lambda+2)\alpha_3, \quad
\alpha_2^* = x^4 \alpha_2 + x^3y \lambda\alpha_3, \quad
\alpha_3^*=x^5 \alpha_3.\]

Assuming  $\alpha_3\neq 0,$  we have the following cases:
\begin{enumerate}

\item if  $\alpha_1= 2(\lambda+2)\alpha_2,$ then  choosing $ y=-\frac{x\alpha_2}{\lambda\alpha_3}$, we have the representative $\langle\nabla_3\rangle;$
\item if  $\alpha_1\neq 2(\lambda+2)\alpha_2,$ then choosing $ x= \frac {\alpha_1-2(\lambda+2)\alpha_2}{\alpha_3},$ $ y=-\frac{x\alpha_2}{\lambda\alpha_3}$, we have the representative $\langle \nabla_1+\nabla_3\rangle.$
\end{enumerate}
Hence, we have  the  following new algebras:

\relax
 
\[ \begin{array}{ll lllllll}
{\mathcal N}^6_{25}(\lambda)_{\lambda \neq 0} & : &
e_1 e_1 = e_2 &    e_1e_2=e_3 & e_1 e_3=e_5 & e_1e_4=\lambda e_5 & e_1e_5=-2 e_6 & e_2 e_1=e_4 \\ 
&&  e_2e_3=e_6 & e_2e_4=\lambda e_6 & e_3e_2= - e_6 &  e_4 e_1=-\lambda e_5&  e_5e_1=e_6    \\
{\mathcal N}^6_{26}(\lambda)_{\lambda \neq 0} & : &
e_1 e_1 = e_2 &    e_1e_2=e_3 & e_1 e_3=e_5 & e_1e_4=\lambda e_5 +e_6 & e_1e_5=-2 e_6& e_2 e_1=e_4\\
& &  e_2e_3=e_6 & e_2e_4=\lambda e_6 & e_3e_2= - e_6 & e_4 e_1=-\lambda e_5-e_6 & e_5e_1=e_6  
\end{array} \]

\relax
\subsection{Central extensions of ${\mathcal N}^5_{05}$} 
Let us use the following notations
\[\nabla_1=[\Delta_{12}], \quad \nabla_2=3[\Delta_{15}]-2[\Delta_{24}]+[\Delta_{33}]-[\Delta_{51}].\]

The automorphism group of ${\mathcal N}^5_{05}$ consists of invertible matrices of the form
\[\phi=\begin{pmatrix}
x & 0 & 0  & 0& 0\\
0 & x^2 & 0 & 0& 0\\
y & 0 & x^3 & 0& 0 \\
z & 0 & 0 &  x^4& 0\\
v & xz & -x^2y & 0& x^5\\
\end{pmatrix}. \]

Since
\[ \phi^T\begin{pmatrix}
0 & \alpha_1 &  0 &0 &  3\alpha_2\\
0 & 0 &  0 & -2\alpha_2 & 0 \\
0 & 0 & \alpha_2 & 0 & 0\\
0 & 0 & 0 & 0 & 0\\
-\alpha_2 & 0 & 0 & 0 & 0
\end{pmatrix}\phi =
\begin{pmatrix}
\alpha^{***} & \alpha_1^* &  -\alpha^{*} &0 &  3\alpha_2^*\\
\alpha^{**} & 0 &  0 & -2\alpha_2^* & 0 \\
\alpha^{*} & 0 & \alpha_2^* & 0 & 0\\
0 & 0 & 0 & 0 & 0\\
-\alpha_2^* & 0 & 0 & 0 & 0
\end{pmatrix},
\]
 we have that the action of ${\rm Aut} ({\mathcal N}^5_{05})$ on the subspace
$\langle \sum\limits_{i=1}^2 \alpha_i\nabla_i  \rangle$
is given by
$\langle \sum\limits_{i=1}^2 \alpha_i^* \nabla_i \rangle,$
where
\[\alpha_1^*=x^3 \alpha_1  + 3 x^2z \alpha_2, \quad
\alpha_2^* =x^6\alpha_2. \]

It is easy to see that we have only one non-trivial orbit with the representative $\langle \nabla_2 \rangle$ and obtain the algebra:
\begin{longtable}{lllllllll}
${\mathcal N}^6_{27}$ &:&
$e_1 e_1 = e_2$ &    $e_1 e_3=e_4$ & $e_1e_4=2e_5$ & $e_1e_5=3e_6$ 
& $e_2 e_1=e_3$ & $e_2e_3=-e_5$ \\ & & $e_2e_4=-2e_6$ & $e_3 e_1=-e_4$  &
$e_3e_3=e_6$ & $e_4e_1=-e_5$ & $e_5e_1=-e_6$ 
\end{longtable}

\subsection{Central extensions of ${\mathcal N}^5_{06}$} Let us use the following notations
\[\nabla_1=[\Delta_{12}], \quad \nabla_2=3[\Delta_{15}]-[\Delta_{22}]-2[\Delta_{24}]-2[\Delta_{31}]+[\Delta_{33}]-[\Delta_{51}].\]

The automorphism group of ${\mathcal N}^5_{06}$ consists of invertible matrices of the form
\[\phi_1=\begin{pmatrix}
1 & 0 & 0  & 0& 0\\
0 & 1 & 0 & 0& 0\\
x & 0 & 1 & 0& 0 \\
y & 0 & 0 &  1& 0\\
z & y & -x &  0& 1\\
\end{pmatrix}, \quad
\phi_2=\begin{pmatrix}
-1 & 0 & 0  & 0& 0\\
0 & 1 & 0 & 0& 0\\
x & 0 & -1 & 0& 0 \\
y & 0 & 0 &  1& 0\\
z & -y & -x &  0& -1\\
\end{pmatrix}.\]

Since
\[ \phi_i^T\begin{pmatrix}
0 & \alpha_1 &  0 &0 &  3\alpha_2\\
0 & -\alpha_2 &  0 & -2\alpha_2 & 0 \\
-2\alpha_2 & 0 & \alpha_2 & 0 & 0\\
0 & 0 & 0 & 0 & 0\\
-\alpha_2 & 0 & 0 & 0 & 0
\end{pmatrix}\phi_i =
\begin{pmatrix}
\alpha^{***} & \alpha_1^* &  \alpha^* &0 &  3\alpha_2^*\\
\alpha^{**} & -2\alpha^*_2 &  0 & -2\alpha_2^* & 0 \\
-2\alpha^*_2-\alpha^* & 0 & \alpha_2^* & 0 & 0\\
0 & 0 & 0 & 0 & 0\\
-\alpha_2^* & 0 & 0 & 0 & 0
\end{pmatrix},
\]
 we have that the action of ${\rm Aut} ({\mathcal N}^5_{06})$ on the subspace
$\langle \sum\limits_{i=1}^2 \alpha_i\nabla_i  \rangle$
is given by
$\langle \sum\limits_{i=1}^2 \alpha_i^* \nabla_i \rangle,$
where
\[\alpha_1^*= (-1)^i\alpha_1 + 3 y \alpha_2, \quad
\alpha_2^* =\alpha_2. \]
It is easy to see  that we have only one non-trivial orbit with the representative $\langle \nabla_2 \rangle$ and obtain the algebra:
\begin{longtable}{ll lllll}
${\mathcal N}^6_{28}$ & : &
$e_1 e_1 = e_2$ &    $e_1e_2=e_5$ & $e_1 e_3=e_4$ & $e_1e_4=2e_5$& $e_1e_5=3e_6$\\
&& $e_2 e_1=e_3$ & $e_2e_2=-e_6$ & $e_2e_3=-e_5$ & $e_2e_4=-2e_6$ &  $e_3 e_1=-e_4-2e_6$\\
&& $e_3e_3=e_6$  & $e_4e_1=-e_5$ & $e_5e_1=-e_6$
\end{longtable}

\subsection{Central extensions of ${\mathcal N}^5_{07}(\lambda)_{\lambda\notin \{0, 1\}}$} Let us use the following notations
\[\nabla_1=[\Delta_{21}], \quad \nabla_2= (4-3\lambda)[\Delta_{15}]+\lambda(3-2\lambda)[\Delta_{24}]+\lambda(2-\lambda)[\Delta_{33}]+\lambda[\Delta_{42}]+\lambda[\Delta_{51}].\]

The automorphism group of ${\mathcal N}^5_{07}(\lambda)_{\lambda\not \in \{0,1\}}$ consists of invertible matrices of the form
\[\phi=\begin{pmatrix}
x & 0 & 0  & 0& 0\\
0 & x^2 & 0 & 0& 0\\
0 & 0& x^3 & 0& 0 \\
y & 0 & 0 &  x^4& 0\\
z & (3-\lambda)xy & 0 &  0& x^5\\
\end{pmatrix}.\]

Since
 \[ \phi^T\begin{pmatrix}
0 & 0 &  0 &0 &  (4-3\lambda)\alpha_2\\
\alpha_1 & 0 &  0 & \lambda(3-2\lambda)\alpha_2 & 0 \\
0 & 0 & \lambda(2-\lambda)\alpha_2 & 0 & 0\\
0 & \lambda\alpha_2 & 0 & 0 & 0\\
\lambda\alpha_2 & 0 & 0 & 0 & 0
\end{pmatrix}\phi =\]
\[=\begin{pmatrix}
\alpha^{**} & \alpha^{*} &  0 &0 &  (4-3\lambda)\alpha_2^*\\
\alpha_1^*+\lambda \alpha^{*} & 0 &  0 & \lambda(3-2\lambda)\alpha_2^* & 0 \\
0 & 0 & \lambda(2-\lambda)\alpha_2^* & 0 & 0\\
0 & \lambda\alpha_2^* & 0 & 0 & 0\\
\lambda\alpha_2^* & 0 & 0 & 0 & 0
\end{pmatrix},
\]
 we have that the action of ${\rm Aut} ({\mathcal N}^5_{07}(\lambda)_{\lambda\not \in \{ 0,1\}})$ on the subspace
$\langle \sum\limits_{i=1}^2 \alpha_i\nabla_i  \rangle$
is given by
$\langle \sum\limits_{i=1}^2 \alpha_i^* \nabla_i \rangle,$
where
\[\alpha_1^*= x^3\alpha_1 -3x^2y\lambda(\lambda-1)(\lambda-2) \alpha_2, \quad
\alpha_2^* =x^6\alpha_2. \]

Assuming $ \alpha_2\neq0$, we have the following cases:
\begin{enumerate}
\item if $\lambda\neq 2,$ then choosing $ y=\frac{x\alpha_1}{3\lambda(\lambda-1)(\lambda-2)\alpha_2}$, we have the representative $\langle \nabla_2\rangle;$
\item if $\lambda=2,$ then we have two representatives $\langle\nabla_2\rangle$ and $\langle\nabla_1+\nabla_2\rangle$ depending on whether $\alpha_1=0$ or not.

\end{enumerate}
Hence, we have the  following new algebras:
\begin{longtable}{ll lllll}
${\mathcal N}^6_{29}(\lambda)_{\lambda \notin \{0, 1\}}$ & : &
$e_1 e_1 = e_2$&   $e_1 e_2=e_3$&    $e_1 e_3=(2-\lambda)e_4$ & $e_1e_4=(3-2\lambda) e_5$  \\
&&
$e_1e_5=(4-3\lambda) e_6$ & $e_2 e_1= \lambda e_3$ & $e_2e_2= \lambda e_4$ & $e_2e_3=\lambda (2-\lambda )e_5$\\
&& $e_2e_4=\lambda (3-2\lambda )e_6$ & $e_3e_1=\lambda e_4$ & $e_3e_2=\lambda e_5$ &  $e_3e_3=\lambda(2-\lambda) e_6$ 
\\
&& $e_4e_1=\lambda e_5$ & $e_4e_2=\lambda e_6$&  $e_5e_1=\lambda e_6$ \\

${\mathcal N}^6_{30}$ & : &
$e_1 e_1 = e_2$&   $e_1 e_2=e_3$& $e_1e_4=-e_5$ & $e_1e_5=-2e_6$ &\\
&& $e_2 e_1= 2 e_3 + e_6$ & $e_2e_2= 2 e_4$ & $e_2e_4=- e_6$ & $e_3e_1=2 e_4$ &\\
&& $e_3e_2=2 e_5$ &  $e_4e_1=2 e_5$ & $e_4e_2=2 e_6$&  $e_5e_1=2 e_6$ &
\end{longtable}

\subsection{Central extensions of ${\mathcal N}^5_{07}(0)$} Let us use the following notations
\[\nabla_1=[\Delta_{15}], \quad \nabla_2=[\Delta_{21}], \quad \nabla_3=2[\Delta_{23}]-2[\Delta_{32}]+[\Delta_{41}].\]

The automorphism group of ${\mathcal N}^5_{07}(0)$ consists of invertible matrices of the form
\[\phi=\begin{pmatrix}
x & 0 & 0  & 0& 0\\
y & x^2 & 0 & 0& 0\\
z & xy & x^3 & 0& 0 \\
v & 2xz & 2x^2y &  x^4& 0\\
w & 3xv & 6x^2z &  3x^3y& x^5\\
\end{pmatrix}. \]

Since
\[ \phi^T\begin{pmatrix}
0 &0 & 0 &  0& \alpha_1\\
\alpha_2 & 0 &  2\alpha_3& 0 & 0 \\
0 &-2\alpha_3 & 0 & 0 & 0\\
 \alpha_3 & 0 & 0 & 0 & 0\\
0 & 0 & 0 & 0 & 0
\end{pmatrix}\phi =
\begin{pmatrix}
\alpha^{****} & \alpha^{***} & \alpha^{**} &  \alpha^{*}& \alpha_1^*\\
\alpha_2^* & 0 &  2\alpha_3^*& 0 & 0 \\
0 &-2\alpha_3^* & 0 & 0 & 0\\
 \alpha_3^* & 0 & 0 & 0 & 0\\
0 & 0 & 0 & 0 & 0
\end{pmatrix},
\]
 we have that the action of ${\rm Aut} ({\mathcal N}^5_{07}(0))$ on the subspace
$\langle \sum\limits_{i=1}^3 \alpha_i\nabla_i  \rangle$
is given by
$\langle \sum\limits_{i=1}^3 \alpha_i^* \nabla_i \rangle,$
where
\[\alpha_1^*=x^6\alpha_1, \quad \alpha_2^* =x^3\alpha_2+2x(2xz-y^2)\alpha_3, \quad \alpha_3^* =x^5\alpha_3. \]

Assuming  $ \alpha_1\neq0$, we have the following cases:
\begin{enumerate}
\item if $\alpha_3\neq0,$ then choosing $x= \frac{\alpha_3} {\alpha_1}, y=0, z=-\frac{x\alpha_2}{4\alpha_3},$
we have the representative $\langle \nabla_1+\nabla_3\rangle;$

\item if $\alpha_3=0, \alpha_2\neq0, $ then choosing $x=\sqrt[3]{{\alpha_2} / {\alpha_1}},$
we have the representative $\langle \nabla_1+\nabla_2\rangle;$

\item if $\alpha_3=\alpha_2=0,$ then we have the representative $\langle \nabla_1\rangle.$

\end{enumerate}

Hence, we have the following new algebras:
\begin{longtable}{llllllll}
${\mathcal N}^6_{31}$ &: &
$e_1 e_1 = e_2$&  $e_1 e_2=e_3$& $e_1 e_3=2e_4$& $e_1e_4=3e_5$ & $e_1e_5=e_6$ \\
&& $e_2 e_3=2e_6$& $e_3e_2=-2e_6$ & $e_4e_1=e_6$& \\
${\mathcal N}^6_{32}$ &: &
$e_1 e_1 = e_2$&  $e_1 e_2=e_3$& $e_1 e_3=2e_4$& $e_1e_4=3e_5$ & $e_1e_5=e_6$& $e_2e_1=e_6$\\
${\mathcal N}^6_{29}(0)$ & : &
$e_1 e_1 = e_2$&   $e_1 e_2=e_3$&    $e_1 e_3=2e_4$ & $e_1e_4=3 e_5$ &
$e_1e_5=4 e_6$
\end{longtable}

\subsection{Central extensions of ${\mathcal N}^5_{07}(1)$} Let us use the following notations
\[\nabla_1=[\Delta_{21}], \quad \nabla_2=[\Delta_{15}]+[\Delta_{24}]+[\Delta_{33}]+[\Delta_{42}]+[\Delta_{51}].\]

The automorphism group of ${\mathcal N}^5_{07}(1)$ consists of invertible matrices of the form
\[\phi=\begin{pmatrix}
x & 0 & 0  & 0& 0\\
y & x^2 & 0 & 0& 0\\
z & 2xy & x^3 & 0& 0 \\
v & 2xz+y^2 & 3x^2y &  x^4& 0\\
w & 2xv+2yz & 3x^2z+3xy^2 &  4x^3y& x^5\\
\end{pmatrix}. \]

Since
\[ \phi^T\begin{pmatrix}
0 &0 & 0 &  0& \alpha_2\\
\alpha_1 & 0 & 0 & \alpha_2 & 0 \\
0 &0 & \alpha_2 & 0 & 0\\
 0 & \alpha_2 & 0 & 0 & 0\\
\alpha_2 & 0 & 0 & 0 & 0
\end{pmatrix}\phi =
\begin{pmatrix}
\alpha^{****} &\alpha^{***} & \alpha^{**} &  \alpha^{*} & \alpha_2^*\\
\alpha_1^*+\alpha^{***} & \alpha^{**} &  \alpha^{*} & \alpha_2^* & 0 \\
\alpha^{**} & \alpha^{*} & \alpha_2^* & 0 & 0\\
 \alpha^{*} & \alpha_2^* & 0 & 0 & 0\\
\alpha_2^* & 0 & 0 & 0 & 0
\end{pmatrix},
\]
 we have that the action of ${\rm Aut} ({\mathcal N}^5_{07}(1))$ on the subspace
$\langle \sum\limits_{i=1}^2 \alpha_i\nabla_i  \rangle$
is given by
$\langle \sum\limits_{i=1}^2 \alpha_i^* \nabla_i \rangle,$
where
\[\alpha_1^*=x^3\alpha_1, \quad \alpha_2^* =x^6\alpha_2. \]
Assuming $ \alpha_2\neq0$, we have the following cases:
\begin{enumerate}
\item if $\alpha_1=0,$ then we have the representative $\langle \nabla_2\rangle;$

\item if $\alpha_1\neq 0,$ then choosing $x=\sqrt[3]{{\alpha_1} / {\alpha_2}}$, we have the representative $\langle \nabla_1+\nabla_2\rangle.$

\end{enumerate}

Hence, we have the algebras:
\begin{longtable}{lllllll}
${\mathcal N}^6_{29}(1)$ & : &
$e_1 e_1 = e_2$&   $e_1 e_2=e_3$&    $e_1 e_3=e_4$ & $e_1e_4= e_5$ &
$e_1e_5= e_6$ \\
&& $e_2 e_1= e_3$ & $e_2e_2=  e_4$ & $e_2e_3=e_5$ & $e_2e_4=e_6$ & $e_3e_1= e_4$ \\
&& $e_3e_2= e_5$ &  $e_3e_3= e_6$ & $e_4e_1=e_5$ & $e_4e_2= e_6$&  $e_5e_1= e_6$ \\
${\mathcal N}^6_{33}$ & : &
$e_1 e_1 = e_2$&   $e_1 e_2=e_3$&    $e_1 e_3=e_4$ & $e_1e_4= e_5$ &$e_1e_5=e_6$ \\
&& $e_2 e_1= e_3 + e_6$ &  $e_2e_2=  e_4$ & $e_2e_3=e_5$ & $e_2e_4=e_6$ & $e_3e_1= e_4$ \\
&& $e_3e_2=e_5$ &  $e_3e_3= e_6$ & $e_4e_1= e_5$ & $e_4e_2= e_6$&  $e_5e_1= e_6$ 
\end{longtable}

\subsection{Central extensions of ${\mathcal N}^5_{08}$} Let us use the following notations
\[\nabla_1=[\Delta_{21}], \quad \nabla_2=2[\Delta_{15}]+[\Delta_{22}]+[\Delta_{31}], \quad  \nabla_3=[\Delta_{23}]-[\Delta_{32}]+[\Delta_{41}].\]

The automorphism group of ${\mathcal N}^5_{08}$ consists of invertible matrices of the form
\[\phi_1=\begin{pmatrix}
1 & 0 & 0  & 0& 0\\
x & 1 & 0 & 0& 0\\
y & x & 1 & 0& 0 \\
z & y & x &  1& 0\\
v & x+z & y &  x& 1\\
\end{pmatrix}, \quad
\phi_2=\begin{pmatrix}
-1 & 0 & 0  & 0& 0\\
x & 1 & 0 & 0& 0\\
y & -x & -1 & 0& 0 \\
z & -y & x &  1& 0\\
v & -x-z & y &  -x& -1\\
\end{pmatrix}.\]

Since
\[ \phi_i^T\begin{pmatrix}
0 &0 & 0 &  0& 2\alpha_2\\
\alpha_1 & \alpha_2 &  \alpha_3& 0 & 0 \\
 \alpha_2 &-\alpha_3 & 0 & 0 & 0\\
 \alpha_3 & 0 & 0 & 0 & 0\\
0 & 0 & 0 & 0 & 0
\end{pmatrix}\phi_i =
\begin{pmatrix}
\alpha^{****} & \alpha^{***} & \alpha^{**} &  \alpha^*& 2\alpha_2^*\\
\alpha_1^*+\alpha^* & \alpha_2^* &  \alpha_3^*& 0 & 0 \\
 \alpha_2^* &-\alpha_3^* & 0 & 0 & 0\\
 \alpha_3^* & 0 & 0 & 0 & 0\\
0 & 0 & 0 & 0 & 0
\end{pmatrix},
\]

we have that the action of ${\rm Aut} ({\mathcal N}^5_{08})$ on the subspace
$\langle \sum\limits_{i=1}^3 \alpha_i\nabla_i  \rangle$
is given by
$\langle \sum\limits_{i=1}^3 \alpha_i^* \nabla_i \rangle,$
where

$$\begin{array}{llll}
\mbox{ for } \ i=1:  & \alpha_1^*=\alpha_1-(x^2-2y)\alpha_3, & \alpha_2^* =\alpha_2, & \alpha_3^* =\alpha_3,\\
\mbox{ for } \ i=2:  & \alpha_1^*=-\alpha_1+(x^2+2y)\alpha_3, & \alpha_2^* =\alpha_2, & \alpha_3^* =-\alpha_3.
\end{array}$$

Assuming $ \alpha_2\neq0$ and using $\phi_1$ we have the following cases:
\begin{enumerate}
\item if $\alpha_3=0,$ then we have the family of representatives $\langle \lambda\nabla_1+\nabla_2\rangle;$

\item if $\alpha_3\neq0,$ then choosing $y=\frac{x^2\alpha_3-\alpha_1}{2\alpha_3},$ we have the family of representatives $\langle \nabla_2+\lambda\nabla_3\rangle_{\lambda\neq0}.$
\end{enumerate}

Hence, we have the following new  algebras:
\begin{longtable}{ll lllll}
${\mathcal N}^6_{34}(\lambda)$ & : &
$e_1 e_1 = e_2$ &   $e_1 e_2=e_3$ & $e_1 e_3=e_4$& $e_1e_4=e_5$&\\
&& $e_1e_5=2e_6$& $e_2e_1=e_5 + \lambda e_6$ & $e_2e_2=e_6$& $e_3e_1=e_6$&\\
 ${\mathcal N}^6_{35}(\lambda)_{\lambda \neq 0}$  & : &
$e_1 e_1 = e_2$ &   $e_1 e_2=e_3$ & $e_1 e_3=e_4$& $e_1e_4=e_5$&\\
&& $e_1e_5=2e_6$ & $e_2e_1=e_5$ & $e_2e_2=e_6$ & $e_2e_3=\lambda e_6$&\\
&& $e_3e_1=e_6$ & $e_3e_2=-\lambda e_6$& $e_4e_1=\lambda e_6$ &&
\end{longtable}

Moreover, using the action of $\phi_2$ we have that
${\mathcal N}^6_{34}(\lambda)\cong{\mathcal N}^6_{34}(-\lambda)$ and ${\mathcal N}^6_{35}(\lambda)\cong{\mathcal N}^6_{35}(-\lambda).$

\subsection{Central extensions of ${\mathcal N}^5_{10}$} Let us use the following notations
\[\nabla_1=[\Delta_{12}], \quad \nabla_2=-2[\Delta_{13}]+[\Delta_{15}]+[\Delta_{24}]+[\Delta_{33}]+[\Delta_{42}]+[\Delta_{51}].\]

The automorphism group of ${\mathcal N}^5_{10}$ consists of invertible matrices of the form
\[\phi_1=\begin{pmatrix}
1 & 0 & 0  & 0& 0\\
x & 1 & 0 & 0& 0\\
y & 2x & 1 & 0& 0 \\
z & x^2+2y & 3x &  1& 0\\
v & 2xy+x+2z & 3x^2+3y &  4x& 1\\
\end{pmatrix}, \quad
\phi_2=\begin{pmatrix}
-1 & 0 & 0  & 0& 0\\
x & 1 & 0 & 0& 0\\
y & -2x & -1 & 0& 0 \\
z & x^2-2y & 3x &  1& 0\\
v & 2xy-x-2z & -3x^2+3y &  -4x& -1\\
\end{pmatrix}.\]

Since
\[ \phi_i^T\begin{pmatrix}
0 & \alpha_1 &  -2\alpha_2 &0 &  \alpha_2\\
0 & 0 &  0 & \alpha_2 & 0 \\
0 & 0 & \alpha_2 & 0 & 0\\
0 & \alpha_2 & 0 & 0 & 0\\
\alpha_2 & 0 & 0 & 0 & 0
\end{pmatrix}\phi_i =
\begin{pmatrix}
\alpha^{****} & \alpha_1^*+\alpha^{***} &  -2\alpha_2^{*}+\alpha^{**} & \alpha^{*} &  \alpha_2^*\\
\alpha^{***} & \alpha^{**} &  \alpha^{*} & \alpha_2^* & 0 \\
\alpha^{**} & \alpha^{*} & \alpha_2^* & 0 & 0\\
\alpha^{*} & \alpha_2^* & 0 & 0 & 0\\
\alpha_2^* & 0 & 0 & 0 & 0
\end{pmatrix},
\]
 we have that the action of ${\rm Aut} ({\mathcal N}^5_{10})$ on the subspace
$\langle \sum\limits_{i=1}^2 \alpha_i\nabla_i  \rangle$
is given by
$\langle \sum\limits_{i=1}^2 \alpha_i^* \nabla_i \rangle,$
where
\[\alpha_1^*= (-1)^i\alpha_1 -4x \alpha_2, \quad
\alpha_2^* =\alpha_2. \]

It is easy to see that we have only one non-trivial orbit with representative $\langle \nabla_2 \rangle.$ The corresponding algebra is 
\begin{longtable}{lllllll}
${\mathcal N}^6_{36}$& : &
$e_1 e_1 = e_2$&   $e_1 e_2=e_3$&    $e_1 e_3=e_4-2e_6$ & $e_1e_4= e_5$ & $e_1e_5=e_6$ \\
&& $e_2 e_1= e_3+e_5$ & $e_2e_2= e_4 $& $e_2e_3=e_5$  &$e_2e_4=e_6$&  $e_3e_1=e_4$ \\
&& $e_3e_2=e_5$ & $e_3e_3=e_6$&  $e_4e_1=e_5 $& $e_4e_2=e_6$ & $e_5e_1=e_6$
\end{longtable}

\subsection{Central extensions of ${\mathcal N}^5_{11}(\lambda)$} Let us use the following notations
\[\nabla_1=[\Delta_{13}], \quad \nabla_2=[\Delta_{15}]+[\Delta_{32}], \quad \nabla_3=[\Delta_{14}]-2[\Delta_{23}]+2[\Delta_{32}]-[\Delta_{41}].\]

The automorphism group of ${\mathcal N}^5_{11}(\lambda)$ consists of invertible matrices of the form
\[\phi=\begin{pmatrix}
1 & 0 & 0  & 0& 0\\
x & 1 & 0 & 0& 0\\
y & x & 1 & 0& 0 \\
z & x+2y & 2x &  1& 0\\
v &  x^2+(\lambda+1)y+2z & (\lambda+3)x+4y &  2x& 1\\
\end{pmatrix}.\]

Since
\[ \phi^T\begin{pmatrix}
0 & 0 &  \alpha_1 & \alpha_3 &  \alpha_2\\
0 & 0 &  -2\alpha_3 & 0 & 0 \\
0 & \alpha_2+2\alpha_3 & 0 & 0 & 0\\
-\alpha_3 & 0 & 0 & 0 & 0\\
0 & 0 & 0 & 0 & 0
\end{pmatrix}\phi =
\begin{pmatrix}
\alpha^{****} & \alpha^{***} &  \alpha_1^*+2\alpha^{**}+\lambda \alpha^{*} & \alpha_3^*+2\alpha^{*} &  \alpha_2^*\\
\alpha^{**} & \alpha^{*} &   -2\alpha_3^* & 0 & 0 \\
\alpha^{*} & \alpha_2^*+2\alpha_3^* & 0 & 0 & 0\\
-\alpha_3^* & 0 & 0 & 0 & 0\\
0 & 0 & 0 & 0 & 0
\end{pmatrix},
\]
 we have that the action of ${\rm Aut} ({\mathcal N}^5_{11}(\lambda))$ on the subspace
$\langle \sum\limits_{i=1}^3 \alpha_i\nabla_i  \rangle$
is given by
$\langle \sum\limits_{i=1}^3 \alpha_i^* \nabla_i \rangle,$
where
\[\alpha_1^*= \alpha_1+(3 x + 4 y-2x^2)\alpha_2+(8y+2x-4x^2)\alpha_3
, \quad
\alpha_2^* =\alpha_2, \quad
\alpha_3^* =\alpha_3. \]

Assuming  $\alpha_2\neq0,$ we have the follofing cases 

\begin{enumerate}
    \item if $\alpha_2 \neq - 2\alpha_3,$ then choosing 
$y=-\frac {\alpha_1+x(3\alpha_2+2\alpha_3)-2x^2(\alpha_2+2\alpha_3)}{4(\alpha_2+2\alpha_3)},$
we have the family of orbits with the representatives 
$\langle \nabla_2 +\mu \nabla_3\rangle_{\mu\neq - \frac 1 2};$
\item if $\alpha_2 = - 2\alpha_3,$ then choosing 
$x=-\frac {\alpha_1}{2\alpha_2},$
we have the family of orbits with the representatives
$\langle \nabla_2 - \frac 1 2 \nabla_3\rangle.$
\end{enumerate}
 Thus, we obtain the family of algebras: 
 \begin{longtable}{ll lllll}
${\mathcal N}^6_{37}(\lambda, \mu)$ & :&
$e_1 e_1 = e_2$&   $e_1 e_2=e_3$&    $e_1 e_3=2e_4 + \lambda e_5$&  $e_1e_4=2e_5 +\mu e_6$&\\
&&$e_1e_5=e_6$ & $e_2 e_1= e_4$ & $e_2e_2=e_5$& $e_2e_3=-2\mu e_6$&\\
&& $e_3e_1=e_5$ & $e_3e_2=(1+2\mu) e_6$&  $e_4e_1=-\mu e_6$
\end{longtable}

\subsection{Central extensions of ${\mathcal N}^5_{13}(\lambda)$} Let us use the following notations
\[\nabla_1=[\Delta_{21}],  \quad \nabla_2=(1-2\lambda)[\Delta_{13}]-2[\Delta_{14}]+[\Delta_{15}]-[\Delta_{23}]+[\Delta_{24}]+[\Delta_{33}]+[\Delta_{41}]+[\Delta_{42}]+[\Delta_{51}].\]

The automorphism group of ${\mathcal N}^5_{13}(\lambda)$ consists of invertible matrices of the form
\[\phi=\begin{pmatrix}
1 & 0 & 0  & 0& 0\\
x & 1 & 0 & 0& 0\\
y & 2x & 1 & 0& 0 \\
z & x^2+x+2y & 3x &  1& 0\\
v & \lambda x+2xy-2y+2z & 3x^2-3x+3y &  4x& 1\\
\end{pmatrix}.\]

Since
\[ \phi^T\begin{pmatrix}
0 & \alpha_1 &  (1-2\lambda)\alpha_2 & -2\alpha_2 &  \alpha_2\\
0 & 0 &  -\alpha_2 & \alpha_2 & 0 \\
0 & 0 & \alpha_2 & 0 & 0\\
\alpha_2 & \alpha_2 & 0 & 0 & 0\\
\alpha_2 & 0 & 0 & 0 & 0
\end{pmatrix}\phi = \]
\[\begin{pmatrix}
\alpha^{****} & \alpha_1^*+\alpha^{***}  &  (1-2\lambda)\alpha_2^*+\alpha^{**}-2\alpha^* & -2\alpha_2^*+\alpha^* &  \alpha_2^*\\
\alpha^{***}+\alpha^{**}+\lambda\alpha^*  & \alpha^{**} &  -\alpha_2^*+\alpha^* & \alpha_2^* & 0 \\
\alpha^{**} & \alpha^* & \alpha_2^* & 0 & 0\\
\alpha_2^*+\alpha^* & \alpha_2^* & 0 & 0 & 0\\
\alpha_2^* & 0 & 0 & 0 & 0
\end{pmatrix},
\]
 we have that the action of ${\rm Aut} ({\mathcal N}^5_{13}(\lambda))$ on the subspace
$\langle \sum\limits_{i=1}^2 \alpha_i\nabla_i  \rangle$
is given by
$\langle \sum\limits_{i=1}^2 \alpha_i^* \nabla_i \rangle,$
where
\[\alpha_1^*= \alpha_1  -(\lambda x -x^2+x+y) \alpha_2, \quad
\alpha_2^* =\alpha_2. \]

It is easy to see  that we have only one non-trivial orbit with the representative $\langle \nabla_2 \rangle$ and obtain the family of algebras:
\begin{longtable}{lllllll}
${\mathcal N}^6_{38}(\lambda)$ & :&
$e_1 e_1 = e_2 $& $e_1 e_2=e_3$&  $e_1 e_3=e_4-2e_5+(1-2\lambda)e_6$ & $e_1e_4=e_5-2e_6$& $e_1e_5=e_6$  \\
&& $e_2 e_1= e_3 + e_4+\lambda e_5$ &  $e_2 e_2= e_4$&  $e_2e_3=e_5-e_6$ &$e_2e_4=e_6$&   $e_3 e_1= e_4$  \\
&& $e_3e_2=e_5$& $e_3e_3=e_6$& $e_4e_1=e_5+e_6$ & $e_4e_2=e_6$& $e_5e_1=e_6$ 
\end{longtable}

\

Summarizing the results of the present section
and using the information from Tables B and C from Appendix, we have the following theorem.

\begin{theorem}
Let $\mathcal N$ be a $6$-dimensional complex one-generated nilpotent Novikov algebra.
Then $\mathcal N$ is isomorphic to  ${\mathcal N}_j^6,$ where $j=01, \ldots,38$  (see Table D presented in Appendix).
\end{theorem}

\section*{Appendix}

\begin{longtable}{|lll|}
\hline
 \multicolumn{3}{|c|}{ \mbox{ {\bf \large{Table A.}}
{\it \large{The list of cohomology spaces of  $4$-dimensional one-generated nilpotent Novikov algebras.}}}} \\
\multicolumn{3}{|c|}{  } \\

\hline
${\rm Z^2}\left( {\mathcal N}^4_{01}\right)$ &$=$&$
\left\langle \begin{array}{l} \Delta_{11},\Delta_{12}, \Delta_{21},\Delta_{13},     \Delta_{14}-\Delta_{41}, \Delta_{22}+\Delta_{31}+2\Delta_{41} \end{array}  \right \rangle$\\ 

${\rm B^2}\left( {\mathcal N}^4_{01}\right)$ &$=$&$ 
\left\langle  \begin{array}{l} \Delta_{11},\Delta_{12}, \Delta_{21}  \end{array}    \right\rangle$  \\

${\rm H^2}\left( {\mathcal N}^4_{01}\right)$ &$=$&$ 
\left\langle\begin{array}{l} [\Delta_{13}], [\Delta_{14}]-[\Delta_{41}],   \relax  [\Delta_{22}]+[\Delta_{31}]+2[\Delta_{41}] \end{array}  \right\rangle$ \\
\hline

${\rm Z^2}\left( {\mathcal N}^4_{02}\right)$ &$=$&$ 
\left\langle \begin{array}{l}\Delta_{11},\Delta_{12}, \Delta_{13}-\Delta_{31},  \relax \Delta_{21}, 2\Delta_{14}-\Delta_{23}-\Delta_{41} \end{array} \right\rangle$          \\

${\rm B^2}\left( {\mathcal N}^4_{02}\right)$ &$=$&$
\left\langle \begin{array}{l}\Delta_{11},\Delta_{13}-\Delta_{31}, \Delta_{21}\end{array} \right\rangle$ \\

${\rm H^2}\left( {\mathcal N}^4_{02}\right)$ &$=$&$
\left \langle \begin{array}{l} [\Delta_{12}],2[\Delta_{14}]-[\Delta_{23}]-[\Delta_{41}] \end{array} \right\rangle$ \\

\hline

${\rm Z^2}\left( {\mathcal N}^4_{03}(\lambda)_{\lambda\not \in \{ 0, 1\} }\right)$ &$=$&
$\left\langle\begin{array}{l}\Delta_{11},\Delta_{12}, \Delta_{21},  (2-\lambda)\Delta_{13}+\lambda\Delta_{22}+\lambda\Delta_{31},\\
(3-2\lambda)\Delta_{14}+(2-\lambda)\lambda\Delta_{23}+  \lambda\Delta_{32}+\lambda\Delta_{41}\end{array} \right\rangle$         \\

${\rm B^2}\left( {\mathcal N}^4_{03}(\lambda)_{\lambda \not \in \{ 0, 1\} }\right)$ &$=$&
$\left\langle\begin{array}{l}\Delta_{11},\Delta_{12}+\lambda\Delta_{21},  (2-\lambda)\Delta_{13}+\lambda\Delta_{22}+\lambda\Delta_{31}\end{array} \right\rangle$ \\

${\rm H^2}\left( {\mathcal N}^4_{03}(\lambda)_{\lambda\not \in \{ 0, 1\} }\right)$ &$=$&
$\left\langle\begin{array}{l} [\Delta_{21}], 
(3-2\lambda)\Delta_{14}+(2-\lambda)\lambda\Delta_{23}+\lambda\Delta_{32}+\lambda\Delta_{41}\end{array}\right\rangle $ \\

\hline

${\rm Z^2}\left( {\mathcal N}^4_{03}(0) \right)$ &$=$&  
 $\left\langle \begin{array}{l} \Delta_{11},\Delta_{12},\Delta_{13},\Delta_{14},\Delta_{21}, \relax 2\Delta_{23}-2\Delta_{32}+\Delta_{41}\end{array} \right\rangle$  \\
 
${\rm B^2}\left( {\mathcal N}^4_{03}(0) \right)$ &$=$&  
$\left \langle \begin{array}{l} \Delta_{11}, \Delta_{12},\Delta_{13}  \end{array} \right\rangle$                \\

${\rm H^2}\left( {\mathcal N}^4_{03}(0) \right)$ &$=$&  
$\left\langle
  \begin{array}{l} [\Delta_{14}], [\Delta_{21}],  2[\Delta_{23}]-2[\Delta_{32}]+[\Delta_{41}] \end{array}  \right\rangle$ \\

\hline

${\rm Z^2}\left( {\mathcal N}^4_{03}(1) \right)$ &$=$&  
$\left\langle\begin{array}{l}\Delta_{11},\Delta_{12}, \Delta_{21},  \Delta_{13}+\Delta_{22}+\Delta_{31}, \Delta_{14}+\Delta_{23}+\Delta_{32}+\Delta_{41}\end{array} \right\rangle$ \\

${\rm B^2}\left( {\mathcal N}^4_{03}(1) \right)$ &$=$&  
$\left\langle\begin{array}{l}\Delta_{11},\Delta_{12}+\Delta_{21},   \relax \Delta_{13}+\Delta_{22}+\Delta_{31}\end{array} \right\rangle$                   \\ 

${\rm H^2}\left( {\mathcal N}^4_{03}(1) \right)$ &$=$&  
$\left\langle\begin{array}{l} [\Delta_{21}],   \relax [\Delta_{14}]+[\Delta_{23}]+[\Delta_{32}]+[\Delta_{41}]\end{array}\right\rangle$  \\

\hline

${\rm Z^2}\left( {\mathcal N}^4_{04} \right)$ &$=$&  
$\left\langle \begin{array}{l} \Delta_{11},\Delta_{12}, \Delta_{13}, \Delta_{21}, 2\Delta_{14}+\Delta_{22}+\Delta_{31}, \Delta_{14}-2\Delta_{23}+2\Delta_{32}-\Delta_{41} \end{array}   \right\rangle$  \\

${\rm B^2}\left( {\mathcal N}^4_{04} \right)$ &$=$&  
$\left\langle \begin{array}{l}\Delta_{11},\Delta_{12}, 2\Delta_{13}+\Delta_{21}\end{array}    \right \rangle$            \\

${\rm H^2}\left( {\mathcal N}^4_{04} \right)$ &$=$&  
  $\left\langle   \begin{array}{l}  [\Delta_{13}],
  2[\Delta_{14}]+[\Delta_{22}]+[\Delta_{31}],  \relax
  [\Delta_{14}]-2[\Delta_{23}]+2[\Delta_{32}]-[\Delta_{41}] \end{array}   \right\rangle$  \\

\hline
${\rm Z^2}\left( {\mathcal N}^4_{05} \right)$ &$=$&  
$\left\langle \begin{array}{l} \Delta_{11},\Delta_{12}, \Delta_{21},  \Delta_{13}+\Delta_{22}+\Delta_{31},  -2\Delta_{13}+  \Delta_{14}+\Delta_{23}+\Delta_{32}+\Delta_{41}\end{array}  \right \rangle$   \\        
${\rm B^2}\left( {\mathcal N}^4_{05} \right)$ &$=$&  
$\left\langle  \begin{array}{l} \Delta_{11},\Delta_{12}+\Delta_{21},  \Delta_{13}+\Delta_{21}+ \Delta_{22}+\Delta_{31}   \end{array}    \right\rangle$                     \\

${\rm H^2}\left( {\mathcal N}^4_{05} \right)$ &$=$&  
$\left\langle\begin{array}{l} [\Delta_{21}],  \relax
 -2[\Delta_{13}]+[\Delta_{14}]+[\Delta_{23}]+[\Delta_{32}]+[\Delta_{41}] \end{array}  \right\rangle$ \\
\hline

\end{longtable}

\begin{longtable}{|lllllll|}
\hline
\multicolumn{7}{|c|}{ \mbox{ {\bf Table B.}
{\it The list of $5$-dimensional one-generated nilpotent Novikov algebras.}}} \\
\multicolumn{7}{|c|}{ }
\\
\hline
${\mathcal N}^5_{01}$ &:& $e_1 e_1 = e_2$  & $e_1e_2=e_4$ & $e_1e_3=e_5$ & $e_2 e_1=e_3$ & $e_3e_1=-e_5$   \\
\hline
${\mathcal N}^5_{02}(\lambda)$ &:& $e_1 e_1 = e_2 $ & $e_1 e_2=e_3$ & $e_1e_3=(2-\lambda)e_5$ &&\\
&& $e_2 e_1=\lambda e_3+e_4$&  $e_2e_2=\lambda e_5$ & $e_3e_1=\lambda e_5$ &&\\
\hline
${\mathcal N}^5_{03}(\lambda, \mu)$ & :&
$e_1 e_1 = e_2$ & $e_1 e_2=e_3$&  $e_1e_3=\lambda e_5$& $e_1 e_4= \mu e_5$ & \\ 
&&  $e_2 e_1=e_4$ & $e_2 e_2= e_5$& $e_3e_1= e_5$ & $e_4 e_1= (2- \mu) e_5$ & \\
\hline
${\mathcal N}^5_{04}(\lambda)_{\lambda\neq0}$ & :&
$e_1 e_1 = e_2$ & $e_1 e_2=e_3$&  $e_1e_3= e_5$& $e_1 e_4= \lambda e_5$ & \\ 
&&  $e_2 e_1=e_4$ & $e_4e_1=-\lambda e_5$ & & & \\
\hline
${\mathcal N}^5_{05}$ &:&  
$e_1 e_1 = e_2$ &    $e_1 e_3=e_4$ & $e_1e_4=2e_5$ & $e_2 e_1=e_3$ &\\ 
&&  $e_2e_3=-e_5$ &  $e_3 e_1=-e_4$  & $e_4e_1=-e_5$ && \\
\hline
${\mathcal N}^5_{06}$ & : &
$e_1 e_1 = e_2$ &    $e_1e_2=e_5$ & $e_1 e_3=e_4$ & $e_1e_4=2e_5$ & \\
&& $e_2 e_1=e_3$ & $e_2e_3=-e_5$ &  $e_3 e_1=-e_4$  & $e_4e_1=-e_5$& \\
\hline
${\mathcal N}^5_{07}(\lambda)$ & : &
$e_1 e_1 = e_2$&   $e_1 e_2=e_3$&    $e_1 e_3=(2-\lambda)e_4$ & $e_1e_4=(3-2\lambda) e_5$ & $e_2 e_1= \lambda e_3$\\ 
&& $e_2e_2= \lambda e_4$ & $e_2e_3=\lambda (2-\lambda )e_5$ & $e_3e_1=\lambda e_4$ & $e_3e_2=\lambda e_5$ &  $e_4e_1=\lambda e_5$\\
\hline
${\mathcal N}^5_{08}$ & : & 
$e_1 e_1 = e_2$ &  $e_1 e_2=e_3$ & $e_1 e_3=e_4$& $e_1e_4=e_5$& $e_2e_1=e_5$\\
\hline
${\mathcal N}^5_{09}(\lambda)$ & : & 
$e_1 e_1 = e_2$&   $e_1 e_2=e_3$ & $e_1 e_3=2e_4$ & $e_1e_4= \lambda e_5$ &\\ 
&& $e_2e_3=2e_5$ & $e_3e_2=-2e_5$ & $e_4e_1=e_5$ &&\\
\hline
${\mathcal N}^5_{10}$ & : &
$e_1 e_1 = e_2$&   $e_1 e_2=e_3$&    $e_1 e_3=e_4$ & $e_1e_4= e_5$ & $e_2 e_1= e_3+e_5$\\ 
&&  $e_2e_2= e_4$ & $e_2e_3=e_5$ & $e_3e_1=e_4$ & $e_3e_2=e_5$ & $e_4e_1=e_5$ \\
\hline
${\mathcal N}^5_{11}(\lambda)$ & : &
$e_1 e_1 = e_2$ &   $e_1 e_2=e_3$ &    $e_1 e_3=2e_4 + \lambda e_5$ &  $e_1e_4=2e_5$ &\\ 
&& $e_2 e_1= e_4$ &  $e_2e_2=e_5$ & $e_3e_1=e_5$ &&\\
\hline
${\mathcal N}^5_{12}(\lambda)$ & : &
$e_1 e_1 = e_2$&   $e_1 e_2=e_3$&    $e_1 e_3=2e_4$ & $e_1e_4=(2\lambda+1) e_5$ & $e_2 e_1= e_4$\\ 
&&  $e_2e_2=\lambda e_5$ & $e_2e_3=-2e_5$ & $e_3e_1=\lambda e_5$ & $e_3e_2=2e_5$ & $e_4e_1=-e_5$\\
\hline
${\mathcal N}^5_{13}(\lambda)$ & :& 
$e_1 e_1 = e_2$ & $e_1 e_2=e_3$&  $e_1 e_3=e_4-2e_5$ & $e_1e_4=e_5$ & $e_2 e_1= e_3 + e_4+\lambda e_5$\\ 
&&  $e_2 e_2= e_4$ & $e_2e_3=e_5$ & $e_3 e_1= e_4$ & $e_3e_2=e_5$ & $e_4e_1=e_5$\\
\hline
\end{longtable}
$${\mathcal N}^5_{04}(0) \cong {\mathcal N}^5_{02}(0).$$

\begin{longtable}{|lll|}
\hline
 \multicolumn{3}{|c|}{ \mbox{ {\bf \large{Table C.}}
{\it \large{The list of cohomology spaces of  $5$-dimensional one-generated nilpotent Novikov algebras.}}}} \\
\multicolumn{3}{|c|}{  } \\

\hline
${\rm Z^2}\left( {\mathcal N}^5_{01}\right)$ &$=$&
$\left\langle \begin{array}{l}
\Delta_{11},\Delta_{12},\Delta_{13}-\Delta_{31}, \Delta_{14},  
\Delta_{21}, 2\Delta_{13}+\Delta_{22}+\Delta_{41},  2\Delta_{15}-\Delta_{23}-\Delta_{51} \end{array} \right\rangle$          \\

${\rm B^2}\left( {\mathcal N}^5_{01}\right)$ &$=$&
$\left\langle  \begin{array}{l}\Delta_{11}, \Delta_{12},\Delta_{13}-\Delta_{31}, \Delta_{21} \end{array} \right\rangle$\\

${\rm H^2}\left( {\mathcal N}^5_{01}\right)$ &$=$&
$\left \langle \begin{array}{l} [\Delta_{14}],  2[\Delta_{13}]+[\Delta_{22}]+[\Delta_{41}],   2[\Delta_{15}]-[\Delta_{23}]-[\Delta_{51}]\relax \end{array} \right\rangle$   \\
\hline

${\rm Z^2}\left( {\mathcal N}^5_{02}(\lambda)_{\lambda \neq 0}\right)$ &$=$&
$\left\langle
\begin{array}{l} \Delta_{11},\Delta_{12}, \Delta_{21},\Delta_{14}-\Delta_{41}, \Delta_{13}-\lambda\Delta_{41}, \Delta_{22}+\Delta_{31}+(2-\lambda)\Delta_{41},\\
 (3-2\lambda)\Delta_{15}+\lambda(2-\lambda)\Delta_{23}+\lambda\Delta_{32}+\lambda\Delta_{51}\end{array}   \right\rangle$ \\

${\rm B^2}\left( {\mathcal N}^5_{02}(\lambda)_{\lambda \neq 0}\right)$ &$=$&
$\left\langle\begin{array}{l} \Delta_{11},\Delta_{12}, \Delta_{21},  (2-\lambda)\Delta_{13}+\lambda\Delta_{22}+\lambda\Delta_{31} \end{array}   \right\rangle$\\

${\rm H^2}\left( {\mathcal N}^5_{02}(\lambda)_{\lambda \neq 0}\right)$ &$=$&
$\left\langle \begin{array}{l} [\Delta_{14}]-[\Delta_{41}], [\Delta_{13}]-\lambda[\Delta_{41}], (3-2\lambda)[\Delta_{15}]+\lambda(2-\lambda)[\Delta_{23}]+
 \lambda[\Delta_{32}]+\lambda[\Delta_{51}]\end{array}   \right\rangle$  \\
 \hline

${\rm Z^2}\left( {\mathcal N}^5_{02}(0) \right)$ &$=$&
$\left\langle \begin{array}{l} \Delta_{11}, \Delta_{12}, \Delta_{21},\Delta_{14}-\Delta_{41},  \Delta_{13},  2\Delta_{14}+\Delta_{22}+\Delta_{31},   \Delta_{15}, 2\Delta_{23}-2\Delta_{32}+\Delta_{51}\end{array}   \right\rangle$ \\

${\rm B^2}\left( {\mathcal N}^5_{02}(0) \right)$ &$=$&
$ \left\langle\begin{array}{l} \Delta_{11},\Delta_{12}, \Delta_{21}, \Delta_{13} \end{array}   \right\rangle$     \\

${\rm B^2}\left( {\mathcal N}^5_{02}(0) \right)$ &$=$&
$\left\langle \begin{array}{l} [\Delta_{14}]-[\Delta_{41}], 2[\Delta_{14}]+[\Delta_{22}]+[\Delta_{31}],  \relax [\Delta_{15}],  
   2[\Delta_{23}]-2[\Delta_{32}]+[\Delta_{51}]\end{array}   \right\rangle$  \\

\hline

${\rm Z^2}\left( {\mathcal N}^5_{03}(\lambda, \mu) \right)$ &$=$& 
$\left\langle \begin{array}{l}\Delta_{11},\Delta_{12}, \Delta_{21}, \Delta_{13},   \Delta_{14} - \Delta_{41},
\Delta_{22}+\Delta_{31}+2\Delta_{41},
(2\mu-1)\Delta_{15}+\lambda(2-\mu) \Delta_{23}+\\
\mu(2-\mu) \Delta_{24} + \Big(3 -(\lambda+1)(2-\mu) \Big)\Delta_{32}+ 
(2-\mu)\Delta_{42}+(2-\mu)\Delta_{51}  \end{array} \right\rangle$ \\

${\rm B^2}\left( {\mathcal N}^5_{03}(\lambda, \mu) \right)$ &$=$& 
$\left \langle \begin{array}{l} \Delta_{11},\Delta_{12}, \Delta_{21},    \lambda \Delta_{13}+ \mu \Delta_{14} +\Delta_{22}  +\Delta_{31} + (2-\mu)\Delta_{41}  \end{array} \right \rangle$ \\

${\rm H^2}\left( {\mathcal N}^5_{03}(\lambda, \mu) \right)$ &$=$& 
$\left \langle \begin{array}{l} [\Delta_{13}],   [\Delta_{14}] - [\Delta_{41}],  (2\mu-1)[\Delta_{15}]+\lambda(2-\mu) [\Delta_{23}]+
\mu(2-\mu) [\Delta_{24}] +\\ \Big(3 -(\lambda+1)(2-\mu) \Big)[\Delta_{32}]+ (2-\mu)[\Delta_{42}]+(2-\mu)[\Delta_{51}] \end{array} \right\rangle$
 \\
 
\hline

${\rm Z^2}\left( {\mathcal N}^5_{04}(\lambda)_{\lambda \neq 0} \right)$ &$=$&  
$\left\langle \begin{array}{l} \Delta_{11}, \Delta_{12}, \Delta_{13}, \Delta_{21},  \Delta_{14} - \Delta_{41},
 \Delta_{22}+\Delta_{31}+2\Delta_{41}, -2\Delta_{15}+ \Delta_{23}+\lambda\Delta_{24}-\Delta_{32}+\Delta_{51} \end{array} \right\rangle$ \\

${\rm B^2}\left( {\mathcal N}^5_{04}(\lambda)_{\lambda \neq 0} \right)$ &$=$&  
  $\left\langle \begin{array}{l}\Delta_{11},\Delta_{12}, \Delta_{21}, \Delta_{13}+ \lambda(\Delta_{14} - \Delta_{41}) \end{array}\right\rangle$ \\

${\rm H^2}\left( {\mathcal N}^5_{04}(\lambda)_{\lambda \neq 0} \right)$ &$=$&  
$\left\langle \begin{array}{l} \relax  [\Delta_{14}] - [\Delta_{41}],  [\Delta_{22}]+[\Delta_{31}]+2[\Delta_{41}] -2[\Delta_{15}]+ [\Delta_{23}]+\lambda[\Delta_{24}]-[\Delta_{32}]+[\Delta_{51}]   \end{array} \right\rangle$
  \\
  
 \hline
${\rm Z^2}\left( {\mathcal N}^5_{04}(0) \right)$ &$=$&  
$\left\langle
\begin{array}{l} \Delta_{11},\Delta_{12}, \Delta_{13},\Delta_{15},\Delta_{21}, \Delta_{14}-\Delta_{41},2\Delta_{14}+\Delta_{22}+\Delta_{31},  2\Delta_{23}-2\Delta_{32}+\Delta_{51}\end{array}   \right\rangle$ \\

${\rm B^2}\left( {\mathcal N}^5_{04}(0) \right)$ &$=$&  
$\left\langle  \begin{array}{l}\Delta_{11}, \Delta_{12},\Delta_{13}, \Delta_{21}  \end{array}\right\rangle$   \\                  

${\rm H^2}\left( {\mathcal N}^5_{04}(0) \right)$ &$=$&  
$\left\langle
\begin{array}{l} [\Delta_{15}], [\Delta_{14}]-[\Delta_{41}],  2[\Delta_{14}]+[\Delta_{22}]+[\Delta_{31}],  2[\Delta_{23}]-2[\Delta_{32}]+[\Delta_{51}]\end{array}   \right\rangle$ \\

\hline
${\rm Z^2}\left( {\mathcal N}^5_{05} \right)$ &$=$&  
$\left\langle
\begin{array}{l} \Delta_{11},\Delta_{12}, \Delta_{13}-\Delta_{31}, 2\Delta_{14}-\Delta_{23}-\Delta_{41}, \Delta_{21}, 3\Delta_{15}-2\Delta_{24}+\Delta_{33}-\Delta_{51}\end{array}   \right\rangle$  \\

${\rm B^2}\left( {\mathcal N}^5_{05} \right)$ &$=$&  
$\left\langle\begin{array}{l} \Delta_{11},\Delta_{21}, \Delta_{13}-\Delta_{31},  2\Delta_{14}-\Delta_{23}-\Delta_{41} \end{array}   \right\rangle$ \\       

${\rm H^2}\left( {\mathcal N}^5_{05} \right)$ &$=$&  
$\left\langle\begin{array}{l}  [\Delta_{12}], 3[\Delta_{15}]-2[\Delta_{24}]+[\Delta_{33}]-[\Delta_{51}]  \end{array}\right\rangle$ \\

\hline

${\rm Z^2}\left( {\mathcal N}^5_{06} \right)$ &$=$&  
   $\left\langle
\begin{array}{l} \Delta_{11},\Delta_{12}, \Delta_{13}-\Delta_{31},\Delta_{21}, 2\Delta_{14}-\Delta_{23}-\Delta_{41},  3\Delta_{15}-\Delta_{22}-2\Delta_{24}-2\Delta_{31}+\Delta_{33}-\Delta_{51} \end{array}   \right\rangle$       \\

${\rm B^2}\left( {\mathcal N}^5_{06} \right)$ &$=$&  
$\left\langle
\begin{array}{l} \Delta_{11},\Delta_{21}, \Delta_{13}-\Delta_{31}, \Delta_{12}+2\Delta_{14}-\Delta_{23}-\Delta_{41}\end{array}   \right\rangle$ \\

${\rm H^2}\left( {\mathcal N}^5_{06} \right)$ &$=$&  
$\left\langle\begin{array}{l} [\Delta_{12}],  3[\Delta_{15}]-[\Delta_{22}]-2[\Delta_{24}]- 2[\Delta_{31}]+[\Delta_{33}]-[\Delta_{51}] \end{array}  \right\rangle$        \\
\hline

${\rm Z^2}\left( {\mathcal N}^5_{07} (\lambda)_{\lambda\neq0} \right)$ &$=$&  
$\left\langle \begin{array}{l} \Delta_{11},\Delta_{12},\Delta_{21}, (2-\lambda)\Delta_{13}+\lambda\Delta_{22}+\lambda\Delta_{31},
(3-2\lambda)\Delta_{14}+\lambda(2-\lambda)\Delta_{23}+\lambda\Delta_{32}+\lambda\Delta_{41}, \\ (4-3\lambda)\Delta_{15}+\lambda(3-2\lambda)\Delta_{24}+
\lambda(2-\lambda)\Delta_{33}+\lambda\Delta_{42}+\lambda\Delta_{51}
\end{array}   \right\rangle$ \\

${\rm B^2}\left( {\mathcal N}^5_{07} (\lambda)_{\lambda\neq0} \right)$ &$=$&  
$\left\langle\begin{array}{l} \Delta_{11},\Delta_{12}+\lambda\Delta_{21}, (2-\lambda)\Delta_{13}+\lambda\Delta_{22}+\lambda\Delta_{31}, (3-2\lambda)\Delta_{14}+\lambda(2-\lambda)\Delta_{23}+ \lambda\Delta_{32}+\lambda\Delta_{41} \end{array}   \right\rangle$ \\

${\rm H^2}\left( {\mathcal N}^5_{07} (\lambda)_{\lambda\neq0} \right)$ &$=$&  
$\left\langle \begin{array}{l} [\Delta_{21}],  (4-3\lambda)[\Delta_{15}]+  \lambda(3-2\lambda)[\Delta_{24}]+ \lambda(2-\lambda)[\Delta_{33}]+\lambda[\Delta_{42}]+\lambda[\Delta_{51}]
\end{array}   \right\rangle$  \\
\hline

${\rm Z^2}\left( {\mathcal N}^5_{07} (0) \right)$ &$=$&  
$\left\langle \begin{array}{l} \Delta_{11},\Delta_{12}, \Delta_{13},\Delta_{14},\Delta_{15}, \Delta_{21},2\Delta_{23}-2\Delta_{32}+\Delta_{41}\end{array}   \right\rangle$ \\

${\rm B^2}\left( {\mathcal N}^5_{07} (0) \right)$ &$=$& 
$\left\langle \begin{array}{l} \Delta_{11},\Delta_{12}, \Delta_{13},\Delta_{14} \end{array}  \right\rangle$   \\

${\rm H^2}\left( {\mathcal N}^5_{07} (0) \right)$ &$=$&  
$\left\langle  \begin{array}{l}  [\Delta_{15}], [\Delta_{21}],2[\Delta_{23}]-2[\Delta_{32}]+[\Delta_{41}] \end{array}   \right\rangle$ \\

\hline

${\rm Z^2}\left( {\mathcal N}^5_{08}  \right)$ &$=$&  
$\left\langle \begin{array}{l} \Delta_{11},\Delta_{12}, \Delta_{13},\Delta_{14},\Delta_{21}, 2\Delta_{15}+\Delta_{22}+\Delta_{31}, \Delta_{23}-\Delta_{32}+\Delta_{41}
\end{array}   \right\rangle$ \\

${\rm B^2}\left( {\mathcal N}^5_{08}  \right)$ &$=$&  
$\left\langle \begin{array}{l}  \Delta_{11},\Delta_{12}, \Delta_{13},\Delta_{14}+\Delta_{21} \end{array}  \right\rangle$   \\

${\rm H^2}\left( {\mathcal N}^5_{08}  \right)$ &$=$&  
$\left\langle \begin{array}{l} [\Delta_{21}], 2[\Delta_{15}]+[\Delta_{22}]+[\Delta_{31}],  \relax [\Delta_{23}]- [\Delta_{32}]+ [\Delta_{41}] \end{array}   \right\rangle$  \\

\hline

${\rm Z^2}\left( {\mathcal N}^5_{09}(\lambda)  \right)$ &$=$&  
$\left\langle \begin{array}{l} \Delta_{11},\Delta_{12}, \Delta_{13},\Delta_{14},\Delta_{21}, 2\Delta_{23}-2\Delta_{32}+\Delta_{41}\end{array}   \right\rangle$ \\

${\rm B^2}\left( {\mathcal N}^5_{09}(\lambda)  \right)$ &$=$&   
$\left\langle\begin{array}{l} \Delta_{11},\Delta_{12}, \Delta_{13}, \lambda\Delta_{14}+2\Delta_{23}-2\Delta_{32}+\Delta_{41}\end{array}   \right\rangle$   \\

${\rm H^2}\left( {\mathcal N}^5_{09}(\lambda)  \right)$ &$=$&  
$\left\langle \begin{array}{l}  [\Delta_{14}], [\Delta_{21}] \end{array}   \right\rangle $    \\
\hline

${\rm Z^2}\left( {\mathcal N}^5_{10} \right)$ &$=$&  
$\left\langle \begin{array}{l} \Delta_{11},\Delta_{12},\Delta_{21},  \Delta_{13}+\Delta_{22}+\Delta_{31}, \relax \Delta_{14}+\Delta_{23}+\Delta_{32}+\Delta_{41}, \\ \relax -2\Delta_{13}+\Delta_{15}+\Delta_{24}+  \Delta_{33}+\Delta_{42}+\Delta_{51} \end{array}   \right\rangle$ \\

${\rm B^2}\left( {\mathcal N}^5_{10} \right)$ &$=$&  
$\left\langle \begin{array}{l} \Delta_{11},\Delta_{12}+\Delta_{21}, \relax \Delta_{13}+\Delta_{22}+\Delta_{31}, \relax  \Delta_{14}+\Delta_{21}+\Delta_{23}+\Delta_{32}+\Delta_{41}
\end{array}  \right\rangle$ \\

${\rm H^2}\left( {\mathcal N}^5_{10} \right)$ &$=$&  
$\left\langle \begin{array}{l} [\Delta_{12}],  -2[\Delta_{13}]+[\Delta_{15}]+[\Delta_{24}]+    [\Delta_{33}]+[\Delta_{42}]+[\Delta_{51}] \end{array}\right\rangle$ \\
 \hline

${\rm Z^2}\left( {\mathcal N}^5_{11}(\lambda)  \right)$ &$=$&  
$\left\langle \begin{array}{l} \Delta_{11},\Delta_{12}, \Delta_{13}, \Delta_{15}+\Delta_{32},  \Delta_{21}, 2\Delta_{14}+\Delta_{22}+\Delta_{31},   \Delta_{14}-2\Delta_{23}+2\Delta_{32}-\Delta_{41}
\end{array}   \right\rangle$\\

${\rm B^2}\left( {\mathcal N}^5_{11}(\lambda)  \right)$ &$=$&  
$\Big\langle \begin{array}{l} \Delta_{11},\Delta_{12}, 2\Delta_{13}+\Delta_{21},  \lambda\Delta_{13}+2\Delta_{14}+\Delta_{22}+\Delta_{31}  \end{array}   \Big\rangle$\\

${\rm H^2}\left( {\mathcal N}^5_{11}(\lambda)  \right)$ &$=$&  
$\left\langle\begin{array}{l}  [\Delta_{13}], [\Delta_{15}]+[\Delta_{32}],   \relax [\Delta_{14}]-2[\Delta_{23}]  +2[\Delta_{32}]-[\Delta_{41}] \end{array} \right\rangle$  \\
 \hline

${\rm Z^2}\left( {\mathcal N}^5_{12}(\lambda)  \right)$ &$=$&  
$\left\langle
\begin{array}{l} \Delta_{11},\Delta_{12}, \Delta_{13}, \Delta_{21},  \Delta_{14}-2\Delta_{23}+2\Delta_{32}-\Delta_{41},  \Delta_{22}+\Delta_{31}+2\Delta_{14} \end{array}   \right\rangle$\\

${\rm B^2}\left( {\mathcal N}^5_{12}(\lambda)  \right)$ &$=$&  
$\left\langle \begin{array}{l} \Delta_{11},\Delta_{12}, 2\Delta_{13}+\Delta_{21},  (1+2\lambda)\Delta_{14}+\lambda\Delta_{22}-2\Delta_{23}+\lambda\Delta_{31}+2\Delta_{32}-\Delta_{41}  \end{array}   \right\rangle$\\

${\rm H^2}\left( {\mathcal N}^5_{12}(\lambda)  \right)$ &$=$&  
$\left\langle \begin{array}{l}  [\Delta_{13}],   [\Delta_{14}]-2[\Delta_{23}]+ 2[\Delta_{32}]-[\Delta_{41}]
\end{array}   \right\rangle$ \\
   \hline

${\rm Z^2}\left( {\mathcal N}^5_{13}(\lambda)  \right)$ &$=$&  
$\Big\langle \begin{array}{l} \Delta_{11},\Delta_{12},\Delta_{21}, \Delta_{13}+\Delta_{22}+\Delta_{31}, -2\Delta_{13}+\Delta_{14}+\Delta_{23}+\Delta_{32}+\Delta_{41},\\
(1-2\lambda)\Delta_{13}-2\Delta_{14}+\Delta_{15}-\Delta_{23}+ \Delta_{24}+\Delta_{33}+\Delta_{41}+\Delta_{42}+\Delta_{51} \end{array}   \Big\rangle$  \\ 

${\rm B^2}\left( {\mathcal N}^5_{13}(\lambda)  \right)$ &$=$&  
$\left\langle \begin{array}{l}  \Delta_{11},\Delta_{12}+\Delta_{21},  \Delta_{13}+\Delta_{21}+\Delta_{22}+\Delta_{31}, -2\Delta_{13}+\Delta_{14}+\lambda\Delta_{21}+ \Delta_{23}+\Delta_{32}+\Delta_{41}\end{array} \right\rangle$\\

${\rm H^2}\left( {\mathcal N}^5_{13}(\lambda)  \right)$ &$=$&  
$\left\langle \begin{array}{l} 
 [\Delta_{12}],   (1-2\lambda)[\Delta_{13}]- 2[\Delta_{14}]+   [\Delta_{15}]- [\Delta_{23}]+[\Delta_{24}]+ \relax [\Delta_{33}]+ [\Delta_{41}]+[\Delta_{42}]+[\Delta_{51}] \end{array}\right \rangle$\\
  
     \hline
\end{longtable}

{\tiny 
\begin{longtable}{|ll lllll |}

\hline
\multicolumn{7}{|c|}{ \mbox{ \large{{\bf Table D.}}
\large{{\it The list of $6$-dimensional  one-generated nilpotent Novikov algebras}}}} \\
\multicolumn{7}{|c|}{ } \\

\hline ${\mathcal N}^6_{01}$ & : &
$e_1 e_1 = e_2$ &    $e_1e_2=e_3$ & $e_1 e_3=e_5$ & $e_1e_4=e_6$ &\\
&& $e_2 e_1=e_4$ & $e_4 e_1=-e_6$  &  & & \\

\hline ${\mathcal N}^6_{02}(\lambda)$ & : &
$e_1 e_1 = e_2$&   $e_1 e_2=e_3$ &    $e_1 e_3=e_5$ & $e_1e_4=\lambda e_6$ &   \\ 
&& $e_2 e_1= e_4$ & $e_2e_2=  e_6$ & $e_3e_1= e_6$  & $e_4e_1= (2-\lambda)e_6$ & \\

\hline ${\mathcal N}^6_{03}(\lambda, \mu)$ & : & 
$e_1 e_1 = e_2$ &  $e_1 e_2=e_3$ &    $e_1 e_3=\lambda e_5 + \mu e_6$ & $e_1e_4= e_5$ &   \\ 
&& $e_2 e_1= e_4$ & $e_2e_2=  e_6$ & $e_3e_1= e_6$  & $e_4e_1= -e_5+2e_6$ & \\
\hline
${\mathcal N}^6_{04}$ & : &
$e_1 e_1 = e_2$ &    $e_1e_2=e_5$ & $e_1 e_3=e_4$ & $e_1e_4=2e_6$ &\\
&& $e_2 e_1=e_3$ & $e_2e_3=-e_6$  & $e_3 e_1=-e_4$  & $e_4e_1=-e_6$ &\\

\hline ${\mathcal N}^6_{05}(\lambda)$ & : &
$e_1 e_1 = e_2$ &   $e_1 e_2=e_3$ &    $e_1 e_3=(2-\lambda)e_4$ & $e_1e_4=(3-2\lambda) e_5$ & 
$e_2 e_1= \lambda e_3+e_6$  \\ 
&&$e_2e_2= \lambda e_4$ & $e_2e_3=\lambda (2-\lambda )e_5$ & $e_3e_1=\lambda e_4$ & $e_3e_2=\lambda e_5$ &  $e_4e_1=\lambda e_5$ \\
\hline ${\mathcal N}^6_{06}(\lambda)$ & : & 
$e_1 e_1 = e_2$&  $e_1 e_2=e_3$& $e_1 e_3=2e_4$& $e_1e_4=\lambda e_5$& $e_2e_1=e_6$ \\
&& $e_2e_3=2e_5$ & $e_3e_2=-2e_5$ & $e_4e_1=e_5$ & &\\
 \hline ${\mathcal N}^6_{07}$ &: &
$e_1 e_1 = e_2$ &  $e_1 e_2=e_3$& $e_1 e_3=2e_4$& $e_1e_4= e_5$&\\ 
&& $e_2e_1=e_5$ & $e_2e_3=2e_6$ & $e_3e_2=-2e_6$& $e_4e_1=e_6$ & \\

\hline ${\mathcal N}^6_{08}$ & : & 
$e_1 e_1 = e_2$&   $e_1 e_2=e_3$& $e_1 e_3=2e_4$ & $e_1e_4= e_5$ & \\ 
&& $e_2e_3=2e_6$ & $e_3e_2=-2e_6$ & $e_4e_1=e_6$ && \\

\hline ${\mathcal N}^6_{09}$ & : &
$e_1 e_1 = e_2$ &   $e_1 e_2=e_3$ &    $e_1 e_3=2e_4+e_6$ & $e_1e_4=2e_5$ & \\ 
&& $e_2 e_1= e_4$ & $e_2e_2=e_5$& $e_3e_1=e_5$ && \\

\hline ${\mathcal N}^6_{10}(\lambda)$ & : &
$e_1 e_1 = e_2$&   $e_1 e_2=e_3$&    $e_1 e_3=2e_4+e_6$ & $e_1e_4=(2\lambda+1) e_5$ & $e_2 e_1= e_4$ \\ 
&& $e_2e_2=\lambda e_5$ & $e_2e_3=-2e_5$  &   $e_3e_1=\lambda e_5$ & $e_3e_2=2e_5$ & $e_4e_1=-e_5$  \\

\hline ${\mathcal N}^6_{11}(\lambda)$ & : &
$e_1 e_1 = e_2$   & $e_1 e_2=e_3$&    $e_1 e_3=2e_4+\lambda e_6$ & $e_1e_4=e_5+ 2e_6$ & $e_2 e_1= e_4$  \\ 
&&$e_2e_2=e_6$& $e_2e_3=-2e_5$  & $e_3e_1=e_6$ & $e_3e_2=2e_5$ & $e_4e_1=-e_5$  \\

\hline ${\mathcal N}^6_{12}$ & :& 
$e_1 e_1 = e_2$ & $e_1 e_2=e_3$&  $e_1 e_3=e_4-2e_5$ & $e_1e_4=e_5$& $e_2 e_1= e_3 + e_4+e_6$  \\ 
&&  $e_2 e_2= e_4$&  $e_2e_3=e_5$ & $e_3 e_1= e_4$ & $e_3e_2=e_5$& $e_4e_1=e_5$\\

\hline ${\mathcal N}^6_{13}$ &:& $e_1 e_1 = e_2$  & $e_1e_2=e_4$ & $e_1e_3=e_5$ & $e_1e_4=e_6$ &$e_1e_5=2e_6$  \\ 
&& $e_2 e_1=e_3$ & $e_2e_3=-e_6$ & $e_3e_1=-e_5$ & $e_5e_1=-e_6$ &\\

\hline ${\mathcal N}^6_{14}(\alpha)$ &:& 
$e_1 e_1 = e_2$  & $e_1e_2=e_3$ & $e_1e_3=e_5+e_6$ & $e_1e_4=\alpha e_6$ &\\ 
&& $e_1e_5=e_6$ & $e_2 e_1=e_3+e_4$ & $e_2e_2=e_5$ & $e_2e_3=e_6$ &\\ 
&& $e_3e_1=e_5$ & $e_3e_2=e_6$  & $e_4e_1=-(\alpha+1)e_6$ & $e_5e_1=e_6$ & \\


\hline ${\mathcal N}^6_{15}$ &:& 
$e_1 e_1 = e_2$  & $e_1e_2=e_3$ & $e_1e_3=-e_5 + e_6$ & $e_1e_5=-3 e_6$ &\\  
&& $e_2 e_1= 3 e_3+e_4$  & $e_2e_2= 3 e_5$ & $e_2e_3=-3 e_6$ &  $e_3e_1=3 e_5$ &\\
&& $e_3e_2=3 e_6$  & $e_4e_1=-3 e_6$ & $e_5e_1=3 e_6$ &&\\

\hline ${\mathcal N}^6_{16}(\lambda)_{\lambda\neq3}$ &:& 
$e_1 e_1 = e_2$  & $e_1e_2=e_3$ & $e_1e_3=(2-\lambda)e_5 $ & $e_1e_4= e_6 $ &\\ 
& & $e_1e_5=(3-2\lambda)e_6$ & $e_2 e_1=\lambda e_3+e_4$  & $e_2e_2=\lambda e_5$ & $e_2e_3=\lambda(2-\lambda) e_6$ &\\
& & $e_3e_1=\lambda e_5$ & $e_3e_2=\lambda e_6$  & $e_4e_1=- e_6$ & $e_5e_1=\lambda e_6$ &\\

\hline ${\mathcal N}^6_{17}(\alpha)$ &:& 
$e_1 e_1 = e_2$  & $e_1e_2=e_3$ & $e_1e_3=2e_5$ & $e_1e_4= e_6$ &\\ 
&& $e_1e_5=\alpha e_6$ & $e_2 e_1=e_4$ &  $e_2e_3=2e_6$ &  $e_3e_1=e_6$ &\\  
&& $e_3e_2=-2 e_6$  & $e_4e_1=-e_6$ & $e_5e_1= e_6$ &&\\

\hline ${\mathcal N}^6_{18}(\alpha)$ &:& 
$e_1 e_1 = e_2$  & $e_1e_2=e_3$ & $e_1e_3=2e_5$ & $e_1e_4= (\alpha+2)e_6$ & $e_1e_5=e_6$ \\ 
&& $e_2 e_1=e_4$ & $e_2e_2=e_6$ &  $e_3e_1=e_6$ & $e_4e_1=-\alpha e_6$ &
\\

\hline ${\mathcal N}^6_{19}(\alpha, \beta)$ &:& 
$e_1 e_1 = e_2$  & $e_1e_2=e_3$ & $e_1e_3=2e_5$ & $e_1e_4= (\alpha+2)e_6$ & \\ 
&& $e_1e_5=\beta e_6$ & $e_2 e_1=e_4$ & $e_2e_2=e_6$ & $e_2e_3=2e_6$ &\\
&& $e_3e_1=e_6$ & $e_3e_2=-2 e_6$  & $e_4e_1=-\alpha e_6$ & $e_5e_1= e_6$ &\\


\hline ${\mathcal N}^6_{20}(\lambda, \mu)$ & : &
$e_1 e_1 = e_2$ &    $e_1e_2=e_3$ & $e_1 e_3=\lambda e_5$ & $e_1e_4=\mu e_5$ & $e_1e_5=(2\mu-1) e_6$\\
&& $e_2 e_1=e_4$ & $e_2 e_2=e_5$ &  $e_2 e_3=\lambda (2-\mu)e_6$ &  $e_2 e_4=\mu (2-\mu) e_6$ & \\
&& $e_3 e_1= e_5$&  $e_4e_1=(2-\mu)e_5$ & $e_3 e_2=\Big(3-(\lambda+1)(2-\mu) \Big)e_6$ & $e_4 e_2= (2-\mu)e_6$ & $e_5 e_1=(2-\mu)e_6$ \\
\hline
${\mathcal N}^6_{21}(\lambda, \mu)_{\mu\neq2}$ & : &
$e_1 e_1 = e_2$ &    $e_1e_2=e_3$ & $e_1 e_3=\lambda e_5 +e_6$ & $e_1e_4=\mu e_5$ & $e_1e_5=(2\mu-1) e_6$\\
&& $e_2 e_1=e_4$ & $e_2 e_2=e_5$ &  $e_2 e_3=\lambda(2-\mu) e_6$ &  $e_2 e_4=\mu(2-\mu) e_6$ & \\
&& $e_3 e_1= e_5$ &  $e_4e_1=(2-\mu)e_5$ & $e_3 e_2=\Big(3-(\lambda+1)(2-\mu)\Big)e_6$ & $e_4 e_2= (2-\mu) e_6$ & $e_5 e_1= (2-\mu)e_6$ \\
\hline

${\mathcal N}^6_{22}(\lambda)$ & : &
$e_1 e_1 = e_2$ &    $e_1e_2=e_3$ & $e_1 e_3=\lambda e_5$ & $e_1e_4=e_5+e_6$ & $e_1e_5= e_6$\\
&& $e_2 e_1=e_4$ & $e_2 e_2=e_5$ &  $e_2 e_3=\lambda e_6$ &  $e_2 e_4=  e_6$ & \\
&& $e_3 e_1= e_5$&  $e_4e_1= e_5-e_6$ & $e_3 e_2=(2-\lambda)e_6$ & $e_4 e_2= e_6$ & $e_5 e_1=e_6$ \\
\hline

${\mathcal N}^6_{23}(\nu)_{\nu\neq0}$ & : &
$e_1 e_1 = e_2$ &    $e_1e_2=e_3$ & $e_1 e_3= e_5 +  \nu e_6$ & $e_1e_4=e_5+e_6$ & $e_1e_5= e_6$\\
&& $e_2 e_1=e_4$ & $e_2 e_2=e_5$ &  $e_2 e_3= e_6$ &  $e_2 e_4=  e_6$ & \\
&& $e_3 e_1= e_5$&  $e_4e_1= e_5-e_6$ & $e_3 e_2=e_6$ & $e_4 e_2= e_6$ & $e_5 e_1=e_6$ \\
\hline
${\mathcal N}^6_{24}(\lambda)$ & : &
$e_1 e_1 = e_2$ &    $e_1e_2=e_3$ & $e_1 e_3=\lambda e_5$ & $e_1e_4=2 e_5 + e_6$ & $e_1e_5=e_6$\\
&& $e_2 e_1=e_4$ & $e_2 e_2=e_5$ & $e_3 e_1= e_5$&  $e_3e_2=e_6$ &\\
\hline
 ${\mathcal N}^6_{25}(\lambda)_{\lambda \neq 0}$ & : &
$e_1 e_1 = e_2$ &    $e_1e_2=e_3$ & $e_1 e_3=e_5$ & $e_1e_4=\lambda e_5$ & $e_1e_5=-2 e_6$\\
&& $e_2 e_1=e_4$ &  $e_2e_3=e_6$ & $e_2e_4=\lambda e_6$ & $e_3e_2= - e_6$ & \\
&& $e_4 e_1=-\lambda e_5$&  $e_5e_1=e_6$ & & & \\
\hline
${\mathcal N}^6_{26}(\lambda)_{\lambda \neq 0}$ & : &
$e_1 e_1 = e_2$ &    $e_1e_2=e_3$ & $e_1 e_3=e_5$ & $e_1e_4=\lambda e_5 +e_6$ & $e_1e_5=-2 e_6$\\
&& $e_2 e_1=e_4$ &  $e_2e_3=e_6$ & $e_2e_4=\lambda e_6$ & $e_3e_2= - e_6$& \\ 
& & $e_4 e_1=-\lambda e_5-e_6$ & $e_5e_1=e_6$ & & & \\

\hline ${\mathcal N}^6_{27}$ &:&  
$e_1 e_1 = e_2$ &    $e_1 e_3=e_4$ & $e_1e_4=2e_5$ & $e_1e_5=3e_6$ &\\ 
&& $e_2 e_1=e_3$ & $e_2e_3=-e_5$ & $e_2e_4=-2e_6$ & $e_3 e_1=-e_4$  &\\
&& $e_3e_3=e_6$ & $e_4e_1=-e_5$ & $e_5e_1=-e_6$ && \\

\hline ${\mathcal N}^6_{28}$ & : &
$e_1 e_1 = e_2$ &    $e_1e_2=e_5$ & $e_1 e_3=e_4$ & $e_1e_4=2e_5$& $e_1e_5=3e_6$\\ 
&& $e_2 e_1=e_3$ & $e_2e_2=-e_6$ & $e_2e_3=-e_5$ & $e_2e_4=-2e_6$ &  $e_3 e_1=-e_4-2e_6$\\
&& $e_3e_3=e_6$  & $e_4e_1=-e_5$ & $e_5e_1=-e_6$ &&\\

\hline ${\mathcal N}^6_{29}(\lambda)$ & : &
$e_1 e_1 = e_2$&   $e_1 e_2=e_3$&    $e_1 e_3=(2-\lambda)e_4$ & $e_1e_4=(3-2\lambda) e_5$ & 
$e_1e_5=(4-3\lambda) e_6$ \\
&& $e_2 e_1= \lambda e_3$ & $e_2e_2= \lambda e_4$ & $e_2e_3=\lambda (2-\lambda )e_5$ & $e_2e_4=\lambda (3-2\lambda )e_6$ & $e_3e_1=\lambda e_4$ \\
&& $e_3e_2=\lambda e_5$ &  $e_3e_3=\lambda(2-\lambda) e_6$ & $e_4e_1=\lambda e_5$ & $e_4e_2=\lambda e_6$&  $e_5e_1=\lambda e_6$ \\

\hline ${\mathcal N}^6_{30}$ & : &
$e_1 e_1 = e_2$&   $e_1 e_2=e_3$& $e_1e_4=-e_5$ & $e_1e_5=-2e_6$ &\\
&& $e_2 e_1= 2 e_3 + e_6$ & $e_2e_2= 2 e_4$ & $e_2e_4=- e_6$ & $e_3e_1=2 e_4$ &\\
&& $e_3e_2=2 e_5$ &  $e_4e_1=2 e_5$ & $e_4e_2=2 e_6$&  $e_5e_1=2 e_6$ &\\

\hline ${\mathcal N}^6_{31}$ &: & 
$e_1 e_1 = e_2$&  $e_1 e_2=e_3$& $e_1 e_3=2e_4$& $e_1e_4=3e_5$ &\\
&& $e_1e_5=e_6$ & $e_2 e_3=2e_6$& $e_3e_2=-2e_6$ & $e_4e_1=e_6$& \\

 \hline ${\mathcal N}^6_{32}$ &: & 
$e_1 e_1 = e_2$&  $e_1 e_2=e_3$& $e_1 e_3=2e_4$&& \\ 
&& $e_1e_4=3e_5$ & $e_1e_5=e_6$& $e_2e_1=e_6$ && \\
\hline
 ${\mathcal N}^6_{33}$ & : &
$e_1 e_1 = e_2$&   $e_1 e_2=e_3$&    $e_1 e_3=e_4$ & $e_1e_4= e_5$ &$e_1e_5=e_6$ \\
&& $e_2 e_1= e_3 + e_6$ &  $e_2e_2=  e_4$ & $e_2e_3=e_5$ & $e_2e_4=e_6$ & $e_3e_1= e_4$ \\
&& $e_3e_2=e_5$ &  $e_3e_3= e_6$ & $e_4e_1= e_5$ & $e_4e_2= e_6$&  $e_5e_1= e_6$ \\

\hline ${\mathcal N}^6_{34}(\lambda)$ & : & 
$e_1 e_1 = e_2$ &   $e_1 e_2=e_3$ & $e_1 e_3=e_4$& $e_1e_4=e_5$&\\ 
&& $e_1e_5=2e_6$& $e_2e_1=e_5 + \lambda e_6$ & $e_2e_2=e_6$& $e_3e_1=e_6$&\\ 

  \hline ${\mathcal N}^6_{35}(\lambda)_{\lambda \neq 0}$  & : & 
$e_1 e_1 = e_2$ &   $e_1 e_2=e_3$ & $e_1 e_3=e_4$& $e_1e_4=e_5$&\\  
&& $e_1e_5=2e_6$ & $e_2e_1=e_5$ & $e_2e_2=e_6$ & $e_2e_3=\lambda e_6$&\\
&& $e_3e_1=e_6$ & $e_3e_2=-\lambda e_6$& $e_4e_1=\lambda e_6$ &&\\

\hline ${\mathcal N}^6_{36}$& : & 
$e_1 e_1 = e_2$&   $e_1 e_2=e_3$&    $e_1 e_3=e_4-2e_6$ & $e_1e_4= e_5$ & $e_1e_5=e_6$ \\ 
&& $e_2 e_1= e_3+e_5$ & $e_2e_2= e_4 $& $e_2e_3=e_5$  &$e_2e_4=e_6$&  $e_3e_1=e_4$ \\
&& $e_3e_2=e_5$ & $e_3e_3=e_6$&  $e_4e_1=e_5 $& $e_4e_2=e_6$ & $e_5e_1=e_6$  \\

\hline ${\mathcal N}^6_{37}(\lambda, \mu)$ & :& 
$e_1 e_1 = e_2$&   $e_1 e_2=e_3$&    $e_1 e_3=2e_4 + \lambda e_5$&  $e_1e_4=2e_5 +\mu e_6$&\\ 
&&$e_1e_5=e_6$ & $e_2 e_1= e_4$ & $e_2e_2=e_5$& $e_2e_3=-2\mu e_6$&\\
&& $e_3e_1=e_5$ & $e_3e_2=(1+2\mu) e_6$&  $e_4e_1=-\mu e_6$ &&\\

\hline ${\mathcal N}^6_{38}(\lambda)$ & :& 
$e_1 e_1 = e_2 $& $e_1 e_2=e_3$&  $e_1 e_3=e_4-2e_5+(1-2\lambda)e_6$ & $e_1e_4=e_5-2e_6$& $e_1e_5=e_6$  \\ 
&& $e_2 e_1= e_3 + e_4+\lambda e_5$ &  $e_2 e_2= e_4$&  $e_2e_3=e_5-e_6$ &$e_2e_4=e_6$&   $e_3 e_1= e_4$  \\ 
&& $e_3e_2=e_5$& $e_3e_3=e_6$& $e_4e_1=e_5+e_6$ & $e_4e_2=e_6$& $e_5e_1=e_6$  \\
\hline 
\end{longtable}} 

\[\begin{array}{lllll} {\mathcal N}^6_{16}(3) \cong {\mathcal N}^6_{05}(3), &
  {\mathcal N}^6_{21}(\lambda,2)\cong {\mathcal N}^6_{20}(\lambda,2),
 & {\mathcal N}^6_{23}(0)\cong {\mathcal N}^6_{22}(1), & {\mathcal N}^6_{25}(0)\cong {\mathcal N}^6_{06}(-2), \\ {\mathcal N}^6_{26}(0)\cong {\mathcal N}^6_{17}(-2),& {\mathcal N}^6_{35}(0) \cong{\mathcal N}^6_{34}(0), 
& {\mathcal N}^6_{34}(\lambda)\cong{\mathcal N}^6_{34}(-\lambda), 
&  {\mathcal N}^6_{35}(\lambda)\cong{\mathcal N}^6_{35}(-\lambda). \end{array}\]


\end{document}